\documentclass[preprint,11pt]{elsarticle}
\usepackage{geometry}
\usepackage{setspace} 
\makeatletter
 \def\ps@pprintTitle{%
 	\let\@oddhead\@empty
 	\let\@evenhead\@empty
 	\def\@oddfoot{\footnotesize\itshape
 		{} \hfill}%
 	\let\@evenfoot\@oddfoot
 }
\makeatother
\usepackage[unicode]{hyperref}

\usepackage{latexsym}
\usepackage{indentfirst}
\usepackage{amsxtra}
\usepackage{amssymb}
\usepackage{amsthm}
\usepackage{amsmath}
\usepackage[all]{xy}
\usepackage{mathrsfs} 

\usepackage{xcolor}
\usepackage{color}

\usepackage{amsfonts}

\usepackage[capitalise]{cleveref}

\newtheorem{theor}{Theorem}[section]
\newtheorem*{theor*}{Theorem}
\newtheorem{prop}[theor]{Proposition}
\newtheorem{lemma}[theor]{Lemma}
\newtheorem{cor}[theor]{Corollary}
\newtheorem*{cor*}{Corollary}
\theoremstyle{definition}               
\newtheorem{defin}[theor]{Definition}
\newtheorem{ex}{Example}
\newtheorem{exs}[ex]{Examples}
\newtheorem{rem}[theor]{Remark}
\newtheorem{que}{Question}

\DeclareMathOperator{\Sym}{Sym}
\DeclareMathOperator{\Aut}{Aut}
\DeclareMathOperator{\End}{End}
\DeclareMathOperator{\id}{id}

\DeclareMathOperator{\Soc}{Soc}
\DeclareMathOperator{\Ann}{Ann}

\DeclareMathOperator{\E}{E}

\DeclareMathOperator{\Map}{Map}
\DeclareMathOperator{\Fix}{Fix}

\newcommand{\dr}[1]{\mathcal{D}_r({#1})}

\usepackage[makeroom]{cancel}
 
\usepackage{soul}
\usepackage{tikz}

\begin{document}

\begin{frontmatter}
	\title{{\bf 
Deformed solutions of the Yang-Baxter equation\\ associated to dual weak braces}}
   \author{Marzia~MAZZOTTA}
 \author{Bernard~RYBO{\L}OWICZ}
\author{Paola~STEFANELLI}
 	%

\begin{abstract}

A recent method for acquiring new solutions of the Yang-Baxter equation involves deforming the classical solution associated with a skew brace. In this work, we demonstrate the applicability of this method to a dual weak brace $\left(S,+,\circ\right)$ and prove that all elements generating deformed solutions belong precisely to the set $\mathcal{D}_r(S)=\{z \in S  \mid \forall  a,b \in S \, \, (a+b) \circ z = a\circ z-z+b \circ z\}$, which we term the \emph{distributor of $S$}. We show it is a full inverse subsemigroup of $\left(S, \circ\right)$ and prove it is an ideal for certain classes of braces.
Additionally, we express the distributor of a brace $S$ in terms of the associativity of the operation $\cdot$, with $\circ$ representing the circle or adjoint operation. In this context, $(\mathcal{D}_r(S),+,\cdot)$ constitutes a Jacobson radical ring contained within $S$.
Furthermore, we explore parameters leading to non-equivalent solutions, emphasizing that even deformed solutions by idempotents may not be equivalent.
Lastly, considering $S$ as a strong semilattice $[Y, B_\alpha, \phi_{\alpha,\beta}]$ of skew braces $B_\alpha$, we establish that a deformed solution forms a semilattice of solutions on each skew brace $B_\alpha$ if and only if the semilattice $Y$ is bounded by an element $1$ and the deforming element $z$ lies in $B_1$.

\end{abstract}

\end{frontmatter}

\section*{Introduction}

The history of the Yang-Baxter equation dates back to the 1960s when the equation first appeared and gained the attention of mathematical physicists. The name originates from two outstanding researchers and their papers: Yang and his paper on many-body problems \cite{Yang}, and Baxter and his paper on the eight-vertex lattice model \cite{Baxter}. Since then, different variants of the equation have been found and studied. For example, the classical Yang-Baxter equation, whose connections with simple Lie algebras were studied by Belavin and Drinfel’d in \cite{BeDr82}.  A variant which is of our interest in this work is called the set-theoretic Yang-Baxter equation. Drinfel'd attracted the attention of researchers to this version by including it in his paper \cite[p.~7]{Dr92} and presenting it in the form we use nowadays.
Given a set $S$, a map
$r:S\times S\to S\times S$ is said to be a \emph{set-theoretic solution of the Yang-Baxter equation}, or shortly  \emph{solution}, if it satisfies the identity
\begin{align*}
\left(r\times\id_S\right)
\left(\id_S\times r\right)
\left(r\times\id_S\right)
= 
\left(\id_S\times r\right)
\left(r\times\id_S\right)
\left(\id_S\times r\right).
\end{align*}
Writing $r\left(x,y\right) = \left(\lambda_x\left(y\right), \rho_y\left(x\right)\right)$, with $\lambda_x, \rho_x$ maps from $S$ into itself,  then $r$ is \emph{left (resp. right) non-degenerate} if $\lambda_x\in \Sym_S$ (resp. $\rho_x\in \Sym_S$),  for every $x\in S$,  \emph{non-degenerate} if it is both left and right non-degenerate. After the paper of Drinfel'd, many authors focused their attention on this equation discovering new connections between solutions and various algebraic structures. The early literature on the subject is rich, but to help new readers, we direct them to the papers by Etingof et al., e.g. \cite{ESS99, EtSoGu01, Et03}. In \cite{ESS99}, Etingof, Schedler, and Soloviev studied involutive non-degenerate solutions. Then, those ideas were further pursued by Lu, Yan, and Zhu in \cite{LuYZ00} with a focus on bijective non-degenerate solutions.

In \cite{Ru07}, Rump introduced the algebraic structure of braces to study involutive non-degenerate solutions.  It is worth noting that the axiomatics of the theory of two-sided braces were already studied before by Andrunakievi\v{c} \cite[p.~131]{And48} and Kurosh \cite[p.~69]{Ku74}. \footnote{This fact was unknown to the authors before an anonymous referee reported it.}. Later, skew braces \cite{GuVe17} were introduced by Guarnieri and Vendramin to study bijective non-degenerate solutions. In this context, Bardakov and Gubarev \cite{BaGu22} have proved that every skew brace can be injectively embedded into a Rota-Baxter group.


Recently in \cite{CaMaMiSt22}, the authors introduced weak braces to study not necessarily bijective solutions. We note that a similar approach of weakening the structure was already considered for the quantum Yang-Baxter equations by introducing weak Hopf algebras in \cite{Li98}).
A \emph{weak brace} is a triple $\left(S,+,\circ\right)$ such that $\left(S,+\right)$ and $\left(S,\circ\right)$ are inverse semigroups and the identities
\begin{align*}
   a\circ\left(b + c\right)
    = a\circ b - a + a\circ c \qquad \& \qquad a \circ a^-=-a+a
\end{align*}
are satisfied, for all $a,b,c\in S$, where $-a$ and $a^-$ denote the inverses of $a$ with respect to $+$ and $\circ$, respectively. Clearly, the sets of the idempotents $\E(S,+)$ and $\E(S, \circ)$ coincide. 
In particular, $(S,+)$ is a Clifford semigroup, and if $(S, \circ)$ also is, $S$ is called \emph{dual weak brace}. Skew braces are dual weak braces since in this case the structures  $(S,+)$ and $(S, \circ)$ are groups having the same identity. Moreover, if $(S,+)$ is abelian, then $S$ is a brace. In particular, by \cite[Theorem 2.1]{CaMaSt24}, any dual weak brace is a strong semilattice of a family of skew braces $\{B_\alpha\}_{\alpha\in Y}$ indexed by a semilattice $Y$.\\
Any weak brace $(S, +, \circ)$ gives rise to a solution $r:S\times S\to S\times S$ defined by
\begin{align*}
    r\left(a,b\right)
    = \left(-a + a\circ b, \ \left(-a + a\circ b\right)^-\circ a\circ b\right),
\end{align*}
for all $a,b\in S$, that is close to be bijective (see \cite[Theorem 11]{CaMaMiSt22}). In the particular case of skew braces, such a map $r$ is bijective and non-degenerate. Moreover, $r$ is involutive, i.e., $r^2=\id_{S \times S}$, if and only if $S$ is a brace. In addition, if $S$ is dual, the solution $r$ is the strong semilattice of the solutions $\{r_\alpha\}_{\alpha\in Y}$, where $r_\alpha$ is exactly the solution associated to each skew brace $B_\alpha$, which compose $S$ (see \cite[Theorem 4.1]{CCoSt21} and \cite[Proposition 2.4]{CaMaSt24}).

The paper \cite{DoRy22x} presents a way to assign a new ``deformed" solution to particular elements of skew braces. In the case of the identity element, we get the usual solution $r$ associated to a skew brace. The main motivation to study this family of maps lies in the fact that if one considers a finite skew brace, its identity, and another element giving rise to a deformed solution.
In \cite{DoRy22x}, one can also find the first hint that the two-sided skew braces are crucial in such an investigation. Recall that a skew brace $(S,+, \circ)$ is \emph{two-sided} if $
    \left(a + b\right)\circ c 
    = a\circ c - c + b\circ c
$
holds, for all $a,b,c\in S$ (see \cite[Definition 2.15]{CeSmVe19}).


In this paper, we extend and describe this class of solutions directly in the context of dual weak braces. To this end, in the third section, we introduce and investigate the \emph{(right) distributor} of a dual weak brace $(S, +, \circ)$, namely, the set
\begin{align*}
    \mathcal{D}_r(S)=\{z \in S \, \mid \, \forall \, a,b \in S \quad (a+b) \circ z=a\circ z-z+b \circ z\},
\end{align*}
that we prove to be an inverse subsemigroup of $\left(S,\circ\right)$ such that $\E(S)\subseteq \mathcal{D}_r(S)$. In the special case of a brace $B$, the distributor is a subbrace and coincides with the elements that associate with all the elements in $B$, namely $\mathcal{D}_r(B) = \{z \in B \, \mid \, \forall \, a,b \in B \quad (a\cdot b)\cdot z = a\cdot(b\cdot z) \}$ where $a\cdot b = a\circ b - a - b$, for all $a,b\in B$. In particular,  $\mathcal{D}_r(B)$ is a radical ring contained in $B$ that, more generally, it is not an ideal of $B$. We show that
for some cyclic braces (cf. \cite{Ru07-cyc}) it is.

The main result is contained in \cref{th: deform_dualweak}, where we show that, fixed $z\in S$, the map $r_z:S\times S\to S\times S$ given by
\begin{align*}
    r_z(a,b)=\left(-a \circ z+a \circ b \circ z, \ \left(-a \circ z+a \circ b \circ z\right)^-\circ a \circ b\right),
\end{align*}
for all $a,b \in S$, is a solution if and only if $z \in \mathcal{D}_r(S)$.  We call the map $r_z$ \emph{solution associated to $S$ deformed by $z$}. In such a case, $r_z$ is not bijective in general, but it has a behavior close to bijectivity and non-degeneracy, as we show in more detail in \cref{theor_rropr}. If $S$ is a skew brace and $z=0$ is the identity of the groups, then $r_z$ coincides with the usual solution $r$. However, although any idempotent determines a deformed solution in any dual weak brace, in general, the map $r_e$, with $e \in \E(S)$, does not coincide with $r$. More precisely, $r_e$ and $r$ are not equivalent in the sense of \cite{ESS99}. In this regard, we raise the issue of studying under which conditions on the parameters, two deformed solutions are equivalent and we give partial answers in this sense. In the particular cases of two-sided skew braces, we show that if two parameters $z$ and $w$ are in the same conjugacy classes of the multiplicative group, then $r_z$ and $r_w$ are equivalent. \\
We conclude the paper by proving that a deformed solution $r_z$ on a dual weak brace is a strong semilattice of solutions on each individual skew brace $B_{\alpha}$ if and only if the underlying semilattice is bounded by an element $1$ and $z \in B_1$.

\bigskip

\section{Preliminaries}

This section is devoted to introducing the structure of the weak brace and its properties useful for our treatment.

\medskip

Initially, for the ease of the reader, let us briefly recall some useful notions on inverse semigroups (see \cite{ClPr61, Ho95, Law98, Pe84}, for more details). A semigroup $S$ is called \emph{inverse semigroup} if, for each $a\in S$, there exists a unique element $a^{-1}$ of $S$ such that 
 $a=aa^{-1}a$ and $a^{-1}=a^{-1}aa^{-1}$, called the \emph{inverse} of $a$. The behaviour of inverse elements in an inverse semigroup $S$ is similar to that in a group, since $(a b)^{-1}=b^{-1} a^{-1}$ and $(a^{-1})^{-1}=a$, for all $a,b \in S$. Denote by $\E(S)$ the set of the idempotents of $S$, clearly, $e =e^{-1}$, for every $e \in \E(S)$, and the idempotents of $S$ are exactly the elements $a a^{-1}$ and $a^{-1}a$, for any $a\in S$. An inverse semigroup $S$ such that $aa^{-1}=a^{-1} a$,
for every $a\in S$, is named \emph{Clifford semigroup}. 
Equivalently, a Clifford semigroup $S$ is an inverse semigroup in which the idempotents are central or, according to \cite[Theorem 4.2.1]{Ho95}, it is a strong (lower) semilattice $Y$ of disjoint groups.

\medskip

Below, we recall the definition of weak brace and dual weak brace contained in \cite[Definition 5]{CaMaMiSt22} and in \cite[Definition 2]{CaMaSt24}.
\begin{defin}\label{def:weak}
    Let $S$ be a set endowed with two binary operations $+$ and $\circ$ such that $\left(S,+\right)$ and $\left(S,\circ\right)$ are inverse semigroups. Then, $\left(S, +, \circ\right)$ is said to be a \emph{weak (left) brace} if the following relations
 \begin{align*}
    a\circ\left(b + c\right)
    = a\circ b - a + a\circ c\qquad \&\qquad a\circ a^-
    = - a + a
 \end{align*}
 are satisfied, for all $a,b,c\in S$, where $-a$ and $a^-$ denote the inverses of $a$ with respect to $+$ and $\circ$, respectively.    Moreover, a weak brace $\left(S, +,\circ\right)$ is said to be a \emph{dual weak brace} if $\left(S,\circ\right)$ is a Clifford semigroup. 
\end{defin}

In any weak brace the sets of idempotents $E\left(S,+\right)$ and $E\left(S, \circ\right)$ coincide, thus we simply denote them by $E\left(S\right)$.  
As proved in \cite[Theorem 8]{CaMaMiSt22}, the additive structure of any weak brace is necessarily a Clifford semigroup.  An example of a weak brace that is not dual is contained in \cite[Example 2-3.]{CaMaMiSt22}. In particular, we say that $(S,+, \circ)$ is a \emph{two-sided weak brace} if  $\left(a + b\right)\circ c = a\circ c - c + b\circ c$, for all $a,b,c\in B$. Note that the notion of \emph{two-sided skew brace} can be found in \cite[Definition 2.15]{CeSmVe19}.

Clearly, \emph{skew braces} \cite{GuVe17} are weak braces since the additive and the multiplicative structures are groups with the same identities. Besides, \emph{braces} \cite{Ru07}  are skew braces in which the additive group is abelian. 
Moreover, in any weak brace, $a \circ \left(a^-+b\right)=-a+a \circ b$, for all $a,b \in S$ (see \cite[Proposition 16]{CaMaMiSt22}).
Easy examples of dual weak braces can be obtained starting from any Clifford semigroup $\left(S, \circ\right)$, by setting $a + b:= a\circ b$ or $a + b:= b\circ a$, for all $a,b\in S$. These are the \emph{trivial weak brace} and the \emph{almost trivial weak brace}, respectively.

\medskip

Any dual weak brace is a strong semilattice of skew braces, as we recall below.
\begin{theor}\emph{\cite[Theorem 2.1]{CaMaSt24}}\label{th:CMS-23x}
    Let $(Y,\land)$ be a semilattice and $\left\{B_{\alpha}\ \left|\ \alpha \in Y\right.\right\}$ a family of disjoint skew braces. For each pair $\alpha,\beta$ of elements of $Y$ such that $\beta \leq \alpha$, let $\phi_{\alpha, \beta}:B_{\alpha}\to B_{\beta}$ be a homomorphism of skew braces such that
\begin{enumerate}
    \item for every $\alpha\in Y$ \  $\phi_{\alpha,\alpha} = \id_{B_\alpha}$,
    \item for all $\gamma,\beta,\alpha\in Y$ such that $\gamma\leq \beta\leq \alpha$, 
 \  $\phi_{\beta, \gamma}\phi_{\alpha,\beta}=\phi_{\alpha,\gamma}$. 
\end{enumerate}
Then, set $S:= \mathop{\dot{\bigcup}}\limits_{\alpha\in Y}B_{\alpha}$, the triple $\left(S,+,\circ\right)$ is a dual weak brace where, for all $a\in B_{\alpha}$ and $b\in B_{\beta}$, for all $\alpha,\beta\in Y$, 
$$
    a + b:= \phi_{\alpha\land \beta,\alpha}(a)\, \underset{\alpha\land\beta}{+}\,\phi_{\alpha\land \beta,\beta}(b)
    \qquad \&\qquad 
    a\circ b:= \phi_{\alpha\land \beta,\alpha}(a)\, \underset{\alpha\land\beta}{\circ}\,\phi_{\alpha\land \beta,\beta}(b).
$$
We call such a dual weak brace the \emph{strong semilattice $S$ of skew braces $B_{\alpha}$}, with $\alpha\in Y$, and denote it by 
\mbox{$S = \left[Y,B_\alpha,\phi_{\alpha,\beta}\right]$.}  
Conversely, any dual weak brace can be obtained in this way.
\end{theor} 

Into the specific, given a dual weak brace $S$, the underlying sets of the skew braces $B_\alpha$ that realize $S$ are exactly the underlying sets of the groups composing both the Clifford semigroup $(S, +)$ and the Clifford semigroup  $(S, \circ)$, as shown in the proof of \cite[Theorem 2.1]{CaMaSt24}.


\medskip

To avoid overloading the notation, hereinafter, for all $a\in B_{\alpha}$ and $b\in B_{\beta}$, we will write
$$
a + b=\phi_{\alpha\beta,\alpha}(a) + \phi_{\alpha\beta,\beta}(b)
\qquad \&  \qquad 
a\circ b=\phi_{\alpha\beta,\alpha}(a)\circ \phi_{\alpha\beta,\beta}(b),
$$
thus the two operations of each skew brace $B_\alpha$ will be clear from the context and we denote the operation on the semilattice $Y$ simply by the juxtaposition.
\medskip

The following are easy instances of dual weak braces. 
\begin{ex} Let $(B,+,\circ)$ be a brace and $\{I_\alpha\}_{\alpha\in \mathbb{N}}$ a family of ideals (see \cite[Definition 2.1]{GuVe17}) such that $I_{0}=\{0\}$, and $I_\alpha\subseteq I_{\alpha+1}$, for every $\alpha \in \mathbb{N}$. Then, considering the following sequence of canonical projections of braces
    $$
B\xrightarrow{\pi_1}B/I_1\xrightarrow{\pi_2}\cdots\xrightarrow{\pi_\alpha} B/I_{\alpha}\xrightarrow{\pi_{\alpha+1}}B/I_{\alpha+1}\xrightarrow{\pi_{\alpha+2}}\cdots,
    $$
    we obtain that the strong semilattice $[\mathbb{N},B/I_\alpha ,\pi_{\alpha}]$ 
    is a dual weak brace.
\end{ex}
\medskip

\begin{ex}\label{ex_u}
      Let $Y \subseteq \mathbb{N}$ be a finite set. Then, for every $n \in Y$, $U_n:=\left(U(\mathbb{Z}/2^n\mathbb{Z}), +_1, \circ \right)$ is a two-sided skew brace on the set of units of $\mathbb{Z}/2^n\mathbb{Z}$, with addition defined by $a+_1b:=a-1+b\pmod {2^n}$, for all $a,b\in U\left(\mathbb{Z}/2^{n}\mathbb{Z}\right)$,  and multiplication given by the multiplication modulo $2^n.$ 
  Consider, for all $n,m\in Y$ such that $m\leq n$, the homomorphism $\phi_{n,m}:U_n\to U_m,$ $a\mapsto a\pmod{m}$. Then, 
  $S=[Y, U_n, \phi_{n, m}]$ is a dual weak brace.  
\end{ex}

The motivation for studying such algebraic structures lies mainly in the fact that they give rise to solutions.

\begin{theor}\emph{\cite[Theorem 11]{CaMaMiSt22}}
    Let $\left(S, +, \circ\right)$ be a weak brace. Then, the map  $r:S\times S\to S\times S$ defined by
\begin{align*}
    r\left(a,b\right)
    = \left(-a + a\circ b, \ \left(-a + a\circ b\right)^-\circ a\circ b\right),
\end{align*}
for all $a,b\in S$, is a solution.
\end{theor}  
\noindent Such a map $r$ has a behaviour close to bijectivity since there exists the solution $r^{op}$ associated to the \emph{opposite weak brace}
$\left(S,+^{op}, \circ\right)$ of $S$, where $a+^{op}b:= b + a$, for all $a,b \in S$, such that
  \begin{align*}
      r\, r^{op}\, r = r, \qquad
      r^{op}\, r\, r^{op} = r^{op}, \qquad \&\qquad rr^{op} = r^{op}r,
  \end{align*}
namely, $r$ is a
completely regular element in $\Map(S \times S)$. In particular, if $S$ is a skew brace, then $r^{op}=r^{-1}$, see \cite{KoTr20}. 
It is shown in \cite[Proposition 2.4]{CaMaSt24} that the solution $r$ associated to any dual weak brace $S = \left[Y,B_\alpha,\phi_{\alpha,\beta}\right]$ is the strong semilattice of the bijective solutions $r_\alpha$ associated to any skew brace $B_\alpha$, a construction technique of solutions provided in \cite[Theorem 4.1]{CCoSt21} and that we recall below.

\begin{theor}\emph{\cite[Theorem 4.1]{CCoSt21}}\label{thm:4.1}
Let $(Y, \wedge)$ be a (lower) semilattice, $\{r_{\alpha}\ |\ \alpha\ \in Y\}$ a family of disjoint solutions on each $X_\alpha$ indexed by $Y$,  and for each pair $\alpha, \beta\in Y$ with $\beta\leq \alpha$,  a map $\phi_{\alpha,\beta}:X_{\alpha}\to X_\beta.$ Let $X=\mathop{\dot{\bigcup}}\limits{\alpha\in Y}X_\alpha$ and $r:X\times X\to X\times X$ the map defined by
$$
r(x,y):=r_{\alpha\beta}(\phi_{\alpha,\alpha\beta}(x),\phi_{\beta,\alpha\beta}(y)),
$$
for all $x\in X_\alpha$ and $y\in X_\beta.$ Then $r$ is a solution if the following conditions are satisfied:
\begin{enumerate}
\item $\phi_{\alpha,\alpha}$ is the identity map of $X_\alpha$ for every $\alpha\in Y,$
\item $\phi_{\beta,\gamma}\phi_{\alpha,\beta}=\phi_{\alpha,\gamma}$ for all $\alpha,\beta,\gamma\in Y$ such that $\gamma\leq \beta\leq \alpha,$
\item $(\phi_{\alpha,\beta}\times \phi_{\alpha,\beta})r_{\alpha}=r_{\beta}(\phi_{\alpha,\beta}\times \phi_{\alpha,\beta})$, for all $\alpha, \beta \in Y$ such that $\beta \leq \alpha$.
\end{enumerate}
We call $r$ \emph{strong semilattice of solutions $r_{\alpha}$ indexed by $Y$}.
\end{theor}

\begin{prop}\emph{ \cite[Proposition 2.4]{CaMaSt24}}\label{th:strong-semil-construction}
	Let $S = \left[Y,B_\alpha,\phi_{\alpha,\beta}\right]$ be a dual weak brace and $\left\{r_{\alpha}\ \left|\ \alpha\in Y\right.\right\}$  the family of disjoint solutions on each $B_{\alpha}$, for every $\alpha \in Y$. Then, the solution $r$ associated to $S$ is the strong semilattice of the solutions $r_\alpha$, for every $\alpha \in Y$.	
\end{prop}

Given a dual weak brace $S$, we are used to denote the components of its solution $r$ by introducing the maps $\lambda_a,\rho_b:S\to S$ defined by 
\begin{align*}
    \lambda_a\left(b\right) = - a + a\circ b
    \qquad \&\qquad
    \rho_b\left(a\right)= \lambda_a\left(b\right)^-\circ a \circ b,
\end{align*}
for all $a,b\in S$. 
The components of the map $r^{op}$ are given by
$$\lambda^{op}_{a}\left(b\right) = a\circ b - a = \left(\rho_{a^-}\left(b^-\right)\right)^- \quad \& \quad \rho^{op}_{b}\left(a\right) =\left(a\circ b - a \right)^-\circ a\circ b = \left(\lambda_{b^-}\left(a^-\right)\right)^-.$$
The map $r$ also is close to being non-degenerate, since
\begin{align*}
    \lambda_a\lambda_{a^-}\lambda_a
    = \lambda_{a},
    \qquad   \lambda_{a^-}\lambda_{a}\lambda_{a^-}
    = \lambda_{a^-}, \qquad &\&\qquad
    \lambda_a\lambda_{a^-}
    =\lambda_{a^-}\lambda_{a}, \\  \rho_{a}\rho_{a^-}\rho_{a}
    = \rho_a,
    \qquad
    \rho_{a^-}\rho_{a}\rho_{a^-} = \rho_{a^-},
    \qquad &\&\qquad
    \rho_{a}\rho_{a^-}
    = \rho_{a^-}\rho_{a}, 
\end{align*}
for every $a \in S$. Clearly, if $S$ is a skew brace, such maps are bijective.
By \cite[Lemma 3]{CaMaMiSt22}, it holds that
$\lambda_{a}\left(b\right)\circ\rho_{b}\left(a\right) = a\circ b$,
for all $a,b\in S$. In addition, one has that the map $\lambda:S\to \End\left(S,+\right), a\mapsto \lambda_a$ is a homomorphism of the inverse semigroup $\left(S,\circ\right)$ into the endomorphism semigroup of $\left(S,+\right)$ and the map $\rho:S\to \Map(S), b\mapsto \rho_b$  is a semigroup anti-homomorphism of the inverse semigroup $\left(S,\circ\right)$ to the monoid $\Map(S)$ of the maps from $S$ into itself. 

\medskip

 In the following lemma, we collect some properties that we will use throughout the paper.

\begin{lemma} \emph{(\cite[Lemma 1, Proposition 9]{CaMaMiSt22}, \cite[Lemma 1]{CaMaSt24})} \label{lemma_dual_weak}
    Let $(S,+, \circ)$ be a weak brace. Then, the following hold:
    \begin{enumerate}
    \item \label{lem:1}$a \circ b=a +\lambda_a(b)$,
        \item \label{lem:2}$a+b=a \circ \lambda_{a^-}\left(b\right)$,
        \item \label{lem:3} $\lambda_a(b)=a \circ b \circ \rho_b(a)^-$,
            \item\label{lem:5} $a \circ (-b)=a-a \circ b+a$,
    \end{enumerate}
    for all $a,b \in S$.
\end{lemma}

\noindent By $1.$ and $2.$ in \cref{lemma_dual_weak}, we obtain that any idempotent $e\in\E\left(S\right)$ satisfies the following 
 \begin{align}\label{pro_dual_idemp}
 e + a = e\circ a 
 = \lambda_e\left(a\right), 
\end{align}
for every $a\in S$.
\medskip

\bigskip

\section{Deformed solutions on dual weak braces}
This section aims to describe deformed solutions associated to any dual weak brace, namely solutions obtained by deforming the classical one. These novel solutions have been introduced in the context of skew braces in \cite{DoRy22x}.

\medskip

\begin{theor}\emph{\cite[Theorem 2.4]{DoRy22x}}
    Let $(B,+, \circ)$ be a skew brace and $z \in B$ such that
    \begin{align}\label{abcz}
        \left(a-b+c\right) \circ z= a \circ z - b \circ z+c \circ z,
    \end{align}
    for all $a,b,c \in B$. Then, the map $\check{r}_{z}: B \times B \to B \times B$ given by 
      \begin{align*}
       \check{r}_z(a,b)=\left(a\circ b-a\circ z+z, \ \left(a\circ b-a\circ z+z\right)^{-} \circ a \circ b \right),
    \end{align*}
    for all $a,b \in B$, is a non-degenerate and bijective solution, called \emph{deformed solution by $z$ on $B$}.
\end{theor}
\noindent We denote the components of $\check{r}_z$ by introducing the maps $$\check{\sigma}^z_a(b)=a \circ b-a \circ z + z \qquad \& \qquad \check{\tau}^z_b(a)=\left(a \circ b-a \circ z + z\right)^-\circ a \circ b,$$ for all $a,b \in B$. 

\begin{rem}
Let $(B, +, \circ)$ be a skew brace. Note that $\check{r}_0$ coincides with the inverse solution of the solution $r$ associated to $B$, namely $\check{r}_0=r^{op}$. 
In general, if $z \in B$ satisfies \eqref{abcz}, one can check that  $\check{r}_z^{-1}$ is the map $r_{z^-}: B \times B \to B \times B$ given by
  \begin{align*}
       r_{z^-}(a,b)=\left(-a\circ z^-+a\circ b\circ z^-, \left(-a\circ z^-+a\circ b\circ z^-\right)^{-} \circ a \circ b \right),
    \end{align*}
    for all $a,b \in B$. Such a map $r_{z^-}$ clearly is non-degenerate as $\check{r}_z$ is. Indeed, for all $a, b \in B$, if we  consider the maps $\sigma_a^z, \tau_b^z: B \to B$ defined by
\begin{align*}
\sigma_a^{z}(b) = -a\circ z+a\circ b\circ z
\qquad \& \qquad
\tau_b^{z}(a) = \left(-a\circ z+a\circ b\circ z\right)^{-} \circ a \circ b,
\end{align*}
the components of $r_{z^-}$ are $\sigma_a^{z^-}(b)$ and $\tau_b^{z^-}(a)$, respectively, and are such that
\begin{align*}
  \left(\sigma_a^{z^-}\right)^{-1}(b)&=\sigma_{a^-}^{z}(b)=\left(\check{\tau}_a^z\left(b^-\right) \right)^-\\\left(\tau^{z^-}_b\right)^{-1}(a)&=\tau_{b^-}^{z^{-}}(a)=\left(\check{\sigma}_b^{z^-}\left(a^-\right)\right)^-,
\end{align*}
for all $a,b \in B$. 
\end{rem}

Throughout our work,  we lay the groundwork to prove that the set of elements $z$ which gives rise to a deformed solution is a subgroup of $(B, \circ)$. This fact will allow  studying the map
 \begin{align*}
       r_{z}(a,b)=\left(-a\circ z+a\circ b\circ z, \ \left(-a\circ z+a\circ b\circ z\right)^{-} \circ a \circ b \right),
    \end{align*}
 avoiding the use of $z^-$, that in the case $z=0$ exactly coincides with the usual solution $r$ associated to $B$. For this reason, hereinafter, we will study such a map $r_z$ directly in the context of a dual weak brace $(S, +, \circ)$ and prove that it is a solution.\\
Note that using \eqref{abcz}, it  can be also written as
\begin{align}
\tag{$\star$} \label{star}
    r_{z}\left(a,b\right) = \left(z^-\circ\lambda_{z\circ a}\left(b\right)\circ z,\ z^-\circ\rho_b\left(z\circ a\right)\right),
\end{align}
for all $a,b\in S$. Below, we provide an identity that is equivalent to the relation \eqref{abcz}. 
\begin{lemma}\label{lemmaD}
    Let $(S,+, \circ)$ be a dual weak brace and $z \in S$. Then, \eqref{abcz} is equivalent to
\begin{align}\label{condition_solu}\tag{$D$}
    \left(a+b\right)\circ z=a \circ z - z + b \circ z,
\end{align}
for all $a,b \in S$. 
\begin{proof}
The identity \eqref{condition_solu} can be trivially obtained by \eqref{abcz} and \eqref{pro_dual_idemp} taking $b = z \circ z^-$.
Conversely, if $x\in S$, we get
\begin{align}\label{eq:-x}
    \left(-x\right)\circ z = z - x\circ z + z.
\end{align}
Indeed, $x \circ x^- \circ z = \left(x-x\right)\circ z = x\circ z - z  + \left(-x\right)\circ z$, thus 
$$\left(-x\right)\circ z \underset{\eqref{pro_dual_idemp}}{=} z-x\circ z+x \circ z-z+\left(-x\right)\circ z=z-x \circ z +  x \circ x^- \circ z\underset{\eqref{pro_dual_idemp}}{=}z-x \circ z+z.$$
Now, let $a,b,c\in S$, applying \eqref{eq:-x} and \eqref{pro_dual_idemp}, we obtain
\begin{align*}
    \left(a-b+c\right) \circ z&= a \circ z -z +(-b+c) \circ z=a\circ z -z+(-b) \circ z-z+c \circ z\\
    &=a \circ z - b \circ z+c \circ z,
\end{align*}
i.e., \eqref{abcz} is satisfied. 
\end{proof}
\end{lemma}

\medskip

\begin{rem}\label{le:cond-sol-dual}
 Observe that if $(S, +, \circ)$ is a dual weak brace and $z \in S$, then \eqref{condition_solu} is also equivalent to the equality
\begin{align}\label{cond_zmeno}\tag{$D^{\prime}$}
      (a+b)\circ z= a \circ z+\left(z^-+b\right) \circ z,
      \end{align}
   for all $a,b \in S$. Indeed, if $b \in S$ and \eqref{condition_solu} is satisfied,
then 
$$
\left(z^-+b\right) \circ z=z^-\circ z-z+b \circ z=-z+b \circ z,
$$
   and so \eqref{cond_zmeno} holds.  
   Conversely, if \eqref{cond_zmeno} is satisfied we have that
   \begin{align*}
       \left(z^-+b\right) \circ z\underset{\eqref{pro_dual_idemp}}{=}-z+\left(z^-\circ z \right)\circ z+\left(z^-+b\right) \circ  z=-z+\left(z^-\circ z+b\right) \circ  z\underset{\eqref{pro_dual_idemp}}{=}-z+b \circ z,
   \end{align*}
   i.e., \eqref{condition_solu} holds.
\end{rem}

\medskip

In light of \cref{lemmaD}, we introduce the following set.
\begin{defin}
    Let $(S,+, \circ)$  be a dual weak brace. Then, we call the set
\begin{align*}
    \mathcal{D}_r(S)=\{z \in S \, \mid \, \forall \, a,b \in S \quad (a+b) \circ z=a\circ z-z+b \circ z\},
\end{align*}
\emph{(right) distributor} of $S$. 
 \end{defin}
\noindent It immediately follows by \eqref{pro_dual_idemp} that $\E(S) \subseteq \mathcal{D}_r(S)$.
\begin{rem}\label{two-sided}
    A dual weak brace  $(S,+, \circ)$ is two-sided if and only if $\mathcal{D}_r(S)=S$. 
\end{rem}

\medskip

We aim to show that the map $r_{z}$ is a solution in any dual brace if and only if $z \in \mathcal{D}_r(S)$. In the next section, we will deepen the algebraic structure of $\mathcal{D}_r(S)$. To prove the main result of this section, we need the following preliminary lemma.

\begin{lemma}\label{le:sigma-tau}
    Let $(S, +, \circ)$ be a dual weak brace and $z \in \mathcal{D}_r(S)$. Then, they hold:
    \begin{enumerate}
        \item \label{1lem}$\sigma^{z}_a\left(b\right)\circ\tau^{z}_b(a)= a\circ b\circ z \circ z^-$,     
       \item $\tau^{z}: (S, \circ) \to \Map(S)$ is an anti-homomorphism,
       \item $\sigma_a^{z}(b)=a \circ b \circ b^- \circ \left(a^- \circ z^- + b \right) \circ z$,
    \item \label{44} $\sigma_a^{z}(b) \circ \sigma_a^{z}(b)^-=a \circ a^-+ b \circ b^-+z \circ z^-$,
    \end{enumerate}
    for all $a, b \in S$.    
\begin{proof}
Let $a,b,c \in S$. Initially, by \eqref{star} and \cref{lemmaD}, we have that
$$\sigma^{z}_a\left(b\right)\circ\tau^{z}_b(a)
    =  z^-\circ\lambda_{z\circ a}\left(b\right)\circ\rho_b\left(z\circ a\right)
    = z^- \circ z\circ a \circ b.$$ \\Moreover,
  \begin{align*}
      \tau^{z}_{b\circ c}\left(a\right)
      &= z^-\circ\rho_{b\circ c}\left(z\circ a\right)
      = z^-\circ \rho_c\rho_b\left(z\circ a\right)
      = z^-\circ \rho_c\left(z\circ \tau^{z}_b\left(a\right)\right)=\tau^{z}_c\tau^{z}_b\left(a\right).
  \end{align*}
Furthermore, since, by \eqref{pro_dual_idemp}, $\lambda_a(b)=a \circ b \circ b^- \circ \left(a^-+b\right)$, we obtain
 $$\sigma_a^{z}(b)=  z^- \circ z\circ a \circ b \circ b^- \circ \left(a^- \circ z^- + b \right) \circ z= a \circ b \circ b^- \circ \left(a^- \circ z^- + b \right) \circ z.$$
  Finally, 
  \begin{align*}
   \sigma_a^{z}(b) \circ \sigma_a^{z}(b)^-&=z^-\circ\lambda_{z\circ a}\left(b\right)\circ \lambda_{z\circ a}\left(b\right)^-\circ z\\
   &=z^- \circ z\circ a \circ \left(a^-\circ z^-+b\right)\circ \left(a^-\circ z^-+b\right)^- \circ a^-\circ z^-\circ z\\
   &=a \circ a^- \circ \left(a^-\circ z^-+b-b-a^-\circ z^-\right) \circ z\circ z^- &\mbox{by \eqref{pro_dual_idemp}}\\
   &=a \circ a^-+b\circ b^-+z\circ z^- &\mbox{by \eqref{pro_dual_idemp}}
  \end{align*}
 which completes the proof.
 \end{proof}
\end{lemma}
\medskip

\begin{rem}\label{rem_sol}
 Given a set $X$, a function $r:S \times S\to S\times S, \  (a,b)\mapsto(\sigma_a(b),\tau_b(a))$  is a solution if and only if the following three  equalities hold
\begin{align}
     &\label{first} \sigma_a\sigma_b(c)=\sigma_{\sigma_a\left(b\right)}\sigma_{\tau_b\left(a\right)}\left(c\right)\tag{Y1},\\
    &  \label{second}\sigma_{\tau_{\sigma_b\left(c\right)}\left(a\right)}\tau_c\left(b\right)=\tau_{\sigma_{\tau_b\left(a\right)}\left(c\right)}\sigma_a\left(b\right)\tag{Y2},\\
      &\label{third}\tau_c\tau_b(a)=\tau_{\tau_c\left(b\right)}\tau_{\sigma_b\left(c\right)}\left(a\right)\tag{Y3},
  \end{align} 
for all $a,b,c \in S$. 
\end{rem}

\medskip

\begin{theor}\label{th: deform_dualweak}
Let $S$ be a dual weak brace and $z \in S$. Then, the map $r_{z}:S \times S \to S \times S$ given by
\begin{align*}
       r_{z}(a,b)=\left(-a\circ z+a\circ b\circ z, \left(-a\circ z+a\circ b\circ z\right)^{-} \circ a \circ b \right),
    \end{align*}
    for all $a,b \in S$,
is a solution if and only if $z \in \mathcal{D}_r(S)$. We call such a map $r_z$ \emph{solution associated to $S$ deformed by $z$}.
\begin{proof}
Let $a,b,c \in S$. To prove \eqref{first}, we observe that
\begin{align*}
  \sigma^{z}_a\sigma^{z}_b(c)
  &=-a \circ z +\left(a-a\circ b \circ z+a \circ b\circ c \circ z\right) \circ z&\mbox{by \cref{lemma_dual_weak}-\ref{lem:5}}\\
  &=-a \circ z+a \circ z -z+\left(-a\circ b \circ z+a \circ b\circ c \circ z\right) \circ z&\mbox{$z\in\mathcal{D}_r(S)$}\\
  &=a\circ a^--z +\sigma^{z}_{a \circ b}(c) \circ z &\mbox{by \eqref{pro_dual_idemp}}\\
  &=a\circ a^-+b \circ b^--z +\sigma^{z}_{a \circ b}(c)\circ z&\mbox{by \eqref{pro_dual_idemp}}
   \end{align*} 
   and 
\begin{align*}
&\sigma^{z}_{\sigma^{z}_a\left(b\right)}\sigma^{z}_{\tau^{z}_b\left(a\right)}\left(c\right)\\
&=-\sigma^{z}_a\left(b\right) \circ z+ \left(\sigma^{z}_a\left(b\right)-a\circ b\circ z+a \circ b \circ c \circ z\right) \circ z &\mbox{by \cref{lemma_dual_weak}-\ref{lem:5} \ \& \ \cref{le:sigma-tau}-\ref{lem:1}}\\
&=-\sigma^{z}_a\left(b\right) \circ z+\sigma^{z}_a\left(b\right) \circ z -z +\sigma^{z}_{a \circ b}(c)\circ z &\mbox{$z \in \mathcal{D}_r(S)$}\\
&=\sigma^{z}_a\left(b\right) \circ \sigma^{z}_a\left(b\right)^- -z+\sigma^{z}_{a \circ b}(c) \circ z&\mbox{by \eqref{pro_dual_idemp}}\\
&= a\circ a^-+b \circ b^--z +\sigma^{z}_{a \circ b}(c)\circ z&\mbox{by \cref{le:sigma-tau}-\ref{44}}
\end{align*}  
Besides, \eqref{third} follows by
  \begin{align*}    \tau^{z}_{\tau^{z}_c\left(b\right)}\tau^{z}_{\sigma^{z}_b\left(c\right)}\left(a\right)
      &= \tau^{z}_{\sigma^z_b\left(c\right)\circ\tau^z_c\left(b\right)}\left(a\right)&\mbox{by \cref{le:sigma-tau}-\ref{lem:2}}\\
      &= \tau^{z}_{z^-\circ z\circ b\circ c}\left(a\right)&\mbox{by \cref{le:sigma-tau}-\ref{lem:1}}\\
      &= z^-\circ\rho_{b\circ c}\rho_{z^-\circ z}\left(z\circ a\right)\\
      &= z^-\circ\rho_{b\circ c}\left( z\circ a\right)&\mbox{by \eqref{pro_dual_idemp}}\\
      &= z^-\circ\rho_{c}\rho_{b}\left(z\circ a\right)\\
      &= z^-\circ\rho_{c}\left(z\circ \tau^{z}_b\left(a\right)\right)\\
      &= \tau^{z}_c\tau^{z}_b\left(a\right). \end{align*}
Furthermore,
\begin{align*}
    &\sigma^{z}_{\tau^{z}_{\sigma^{z}_b\left(c\right)}\left(a\right)}\tau^{z}_c\left(b\right)\\
    &=\tau^{z}_{\sigma^{z}_b\left(c\right)}\left(a\right) \circ \tau^{z}_c(b) \circ \tau^{z}_c(b)^-\circ \left(\tau^{z}_{\sigma^{z}_b\left(c\right)}\left(a\right)^- \circ z^- + \tau^{z}_c(b) \right) \circ z &\mbox{by \cref{le:sigma-tau}-\ref{lem:3}}\\
    &=\left(\sigma^{z}_a\sigma^{z}_b(c)\right)^-\circ a \circ \sigma^{z}_b(c) \circ \tau^{z}_c(b) \circ \left(\tau^{z}_{\tau^{z}_c\left(b\right)}\tau^{z}_{\sigma^{z}_b\left(c\right)}\left(a\right) \right)^-&\mbox{by \cref{le:sigma-tau}-\ref{lem:3}}\\
    &=\left(\sigma^{z}_{\sigma^{z}_a\left(b\right)}\sigma^{z}_{\tau^{z}_b\left(a\right)}\left(c\right)\right)^- \circ a \circ \sigma^{z}_b(c) \circ \tau^{z}_c(b)  \circ \left(\tau^{z}_c\tau^{z}_b\left(a\right)\right)^- &\mbox{by (Y1)-(Y3)}\\
    &=\left(\sigma^{z}_{\sigma^{z}_a\left(b\right)}\sigma^{z}_{\tau^{z}_b\left(a\right)}\left(c\right)\right)^- \circ a \circ b \circ c \circ z^- \circ z \circ \left(\tau^{z}_c\tau^{z}_b\left(a\right)\right)^- &\mbox{by \cref{le:sigma-tau}-\ref{lem:1}}\\
    &=\left(\sigma^{z}_{\sigma^{z}_a\left(b\right)}\sigma^{z}_{\tau^{z}_b\left(a\right)}\left(c\right)\right)^- \circ \sigma^{z}_a(b)\circ \tau^{z}_b(a) \circ c \circ  \left(\tau^{z}_c\tau^{z}_b\left(a\right)\right)^- &\mbox{by \cref{le:sigma-tau}-\ref{lem:1}}\\
    &=\left(\sigma^{z}_{\sigma^{z}_a\left(b\right)}\sigma^{z}_{\tau^{z}_b\left(a\right)}\left(c\right)\right)^- \circ \sigma^{z}_a(b)\circ \tau^{z}_b(a) \circ c \circ  c^- \circ \left(\tau^{z}_b(a)^- \circ z^-+c\right) \circ z\\
    &=\left(\sigma^{z}_{\sigma^{z}_a\left(b\right)}\sigma^{z}_{\tau^{z}_b\left(a\right)}\left(c\right)\right)^- \circ \sigma^{z}_a(b)\circ \sigma^{z}_{\tau^{z}_b(a)}(c) &\mbox{by \cref{le:sigma-tau}-\ref{lem:3}}\\
    &=\tau^{z}_{\sigma^{z}_{\tau^{z}_b\left(a\right)}\left(c\right)}\sigma^{z}_a\left(b\right),  
\end{align*}
i.e., \eqref{second} holds.
Therefore, $r_{z}$ is a solution. Finally, we show that \eqref{first} implies \eqref{condition_solu}, namely $z \in \mathcal{D}_r(S)$. Indeed, if $a,b \in S$, by  choosing $x= a$, $y=a^-\circ a$, and $w= a^-\circ (a \circ z+b)\circ z^-$,  we obtain
\begin{align*}
\sigma^{z}_x\sigma^{z}_y(w)
  &=-x \circ z+x \circ (-y \circ z+y \circ w \circ z) \circ z\\
  &=-a\circ z +a \circ \left(-z + a \circ a^- + a^- \circ (a \circ z+b)\circ z^- \circ z\right) \circ z   &\mbox{by \eqref{pro_dual_idemp}}\\
  &= -a \circ z +a \circ \left(-z+a^- \circ a \circ z - a^-+a^-\circ b\right) \circ z &\mbox{by \eqref{pro_dual_idemp}}\\
  &=-a \circ z + a \circ \lambda_{a^-}(b) \circ z &\mbox{by \eqref{pro_dual_idemp}}\\
  &=-a \circ z + (a+b) \circ z &\mbox{by \cref{lemma_dual_weak}-\ref{lem:2}}
\end{align*}
and, using \cref{le:sigma-tau}-\ref{1lem}, \cref{lemma_dual_weak}-\ref{lem:5}, and \eqref{pro_dual_idemp}, we get
\begin{align*}
&\sigma^{z}_{\sigma^{z}_x\left(y\right)}\sigma^{z}_{\tau^{z}_y\left(x\right)}\left(w\right)\\
    &= -\sigma^{z}_x\left(y\right) \circ z+ \left(\sigma^{z}_x\left(y\right)-x\circ y\circ z+x \circ y \circ w \circ z\right) \circ z\\
    &= -\left(-a\circ z+a\circ z\right)\circ z + \left(\left(-a\circ z+a\circ z\right)-a\circ z+a\circ  a^-\circ (a \circ z+b)\circ z^- \circ z \right) \circ z\\
    &=-\left(a \circ a^-\circ z \circ z^-\right)\circ z + \left(-a\circ z + a\circ z - a\circ a^- +a \circ a^-\circ b\right) \circ z\\
    &=-z+\left(a \circ a^-\circ z \circ z^-+b \right) \circ z\\
    &=-z+a\circ a^-+b \circ z.
    \end{align*}
Hence,  by \eqref{first} it follows that
\begin{align}\label{key}
    -a\circ z+(a+b) \circ z=-z+a\circ a^-+b \circ z
\end{align}
and so, by \eqref{pro_dual_idemp}, we can write
\begin{align*}
    (a+b) \circ z=a\circ z-a\circ z+(a+b)\circ z \underset{\eqref{key}}{=}a\circ z-z+a\circ a^-+b \circ z=a\circ z-z+b \circ z,
\end{align*}
i.e., $z \in \mathcal{D}_r(S)$. Therefore, we get the claim.
\end{proof}
\end{theor}

\medskip


The following theorem illustrates some properties of a deformed solution $r_z$ on a dual weak brace $(S,+,\circ)$. In particular, $r_z$ has a behavior close to bijectivity and non-degeneracy since it is completely regular in $\Map(S\times S)$.

\begin{theor}\label{theor_rropr}
  Let $(S,+, \circ)$ be a dual weak brace, $z \in \mathcal{D}_r(S)$, and $r_z$ the solution associated to $S$ deformed by $z$. Then, considered the map $\check{r}_{z^-}: S  \times S \to S \times S$ given by
    \begin{align*}
     \check{r}_{z^-}(a,b)=\left(a\circ b-a\circ z^-+z^-, \ \left(a\circ b-a\circ z^-+z^-\right)^{-} \circ a \circ b \right),
    \end{align*}
    for all $a,b \in S$,  the following hold 
  \begin{align*}
      r_{z}\, \check{r}_{z^-}\, r_{z} = r_{z}, \qquad
      \check{r}_{z^-}\, r_{z}\, \check{r}_{z^-}= \check{r}_{z^-}, \qquad \text{and}\qquad r_{z}\check{r}_{z^-} = \check{r}_{z^-}r_{z}.
  \end{align*}
  Moreover, $\sigma^{z}_a$ and $\tau_a^{z}$ are  completely regular elements in $\Map(S)$, since
  \begin{align*}
\sigma_a^{z}\sigma_{a^-}^{z^-}\sigma_a^{z}
    = \sigma_a^{z},
    \qquad   \sigma_{a^-}^{z^-}\sigma_a^{z}\sigma_{a^-}^{z^-}
    = \sigma_{a^-}^{z^-}, \qquad &\&\qquad
   \sigma_{a}^{z}\sigma_{a^-}^{z^-}
    =\sigma_{a^-}^{z^-}\sigma_{a}^{z} \\ \tau_{a}^{z}\tau^{z}_{a^-}\tau_{a}^{z} = \tau_{a}^{z} 
    , 
    \qquad
    \tau^{z}_{a^-}\tau_{a}^{z}\tau^{z}_{a^-}
    = \tau^{z}_{a^-},
    \qquad &\&\qquad
    \tau^{z}_a \tau_{a^-}^{z}
    = \tau_{a^-}^{z}\tau_a^z, 
\end{align*}
for every $a \in S$. 
  \begin{proof}
  In this proof, we will set for brevity $x^0:=x \circ x^-=x-x$, for every $x \in S$.\\
Initially, if $a,b\in S$,  the first component of \ $r_{z}\check{r}_{z^-}(a,b)$ is equal to
\begin{align*}
    &X:=-\left(a\circ b - a\circ z^- + z^-\right)\circ z + \left(a\circ b - a\circ z^- + z^-\right)^0\circ a\circ b\circ z\\
    &= - \left(a\circ b\circ z - z + z-a\circ z^0 \right) +  \left(a\circ b - a\circ z^- + z^-\right)^0+ a\circ b\circ z &\mbox{by \cref{lemmaD}}\\
    &=a+z^0-a\circ b \circ z+(a\circ b-a \circ z^-+z^0+a\circ z^--a\circ b)+a\circ b \circ z &\mbox{by \eqref{pro_dual_idemp}}\\
    &= a-a\circ b \circ z+(a\circ b)^0+(a\circ z)^0+a\circ b \circ z&\mbox{by \eqref{pro_dual_idemp}}\\
    &=a + b^0 + z^0 &\mbox{by \eqref{pro_dual_idemp}}
\end{align*}
and the first component of  \ $\check{r}_{z^-}r_{z}(a,b)$ is equal to
\begin{align*}
    &Y:=\left(-a\circ z + a\circ b\circ z\right)^0\circ a\circ b - \left(- a\circ z + a\circ b\circ z\right)\circ z^- + z^-\\
    &= \left(-a\circ z + a\circ b\circ z\right)^0 + a\circ b- \left(z^- - a\circ z\circ z^- + a\circ b\right) \circ  z^0 +z^- &\mbox{by \eqref{condition_solu}-\eqref{eq:-x}}\\
    &= (-a\circ z + a\circ b\circ z-a\circ b\circ z+a\circ z)+(a\circ b)^0+a+z^0 &\mbox{by \eqref{pro_dual_idemp}}\\
    &=a + b^0 + z^0 &\mbox{by \eqref{pro_dual_idemp}}
\end{align*}
thus they are equal. Since the second components of $r_{z}\check{r}_{z^-}(a,b)$ and $\check{r}_{z^-}r_{z}(a,b)$ are equal to $X^-\circ a \circ b$ and $Y^- \circ a \circ b$, respectively, and $X = Y$, it follows that $r_{z}\check{r}_{z^-}=\check{r}_{z^-}r_{z}$. Moreover, by the previous paragraph and by \eqref{pro_dual_idemp}, we compute
\begin{align*}
   &r_z\, \check{r}_{z^-}\, r_{z}(a,b)
   = r_{z}\left(Y, Y^- \circ a \circ b\right)\\ 
   &=\left(-Y \circ z + Y \circ Y^-\circ a \circ b \circ z, (-Y \circ z + Y \circ Y^-\circ a \circ b \circ z)^-\circ Y\circ Y^-\circ a\circ b\right)\\ 
   &=r_{z}(a,b)
\end{align*}
and, by \eqref{pro_dual_idemp},
\begin{align*}
  &\check{r}_{z^-}\, r_{z}\, \check{r}_{z^-}(a,b)
  = \check{r}_{z^-}\left(X, X^-\circ a \circ b\right)\\
  &=\left(X \circ X^-\circ a \circ b - X \circ z^-+z^-, \left(X \circ X^-\circ a \circ b - X \circ z^-+z^-\right)^- \circ X \circ X^- \circ a \circ b\right)\\
  &= \check{r}_{z^-}(a,b).
\end{align*}
Furthermore,  it is easy to check that $\sigma_a^{z}\sigma_{a^-}^{z^-}(b)=\sigma_{a^-}^{z^-}\sigma_{a}^{z}(b)=a^0\circ z^0\circ b$ and so they follow $\sigma_a^{z}\sigma_{a^-}^{z^-}\sigma_a^{z}(b)=\sigma_a^{z}(b)$ and $\sigma_{a^-}^{z^-}\sigma_a^{z}\sigma_{a^-}^{z^-}=\sigma_{a^-}^{z^-}(b)$. The rest of the claim is a consequence of \cref{le:sigma-tau}-\ref{lem:2}.
\end{proof}
\end{theor}
\medskip

Observe that if $(B,+, \circ)$ is a skew brace and $a, z\in B$ are such that $a \circ z = z + a$, then $\sigma^z_a = \lambda_a$ (cf. \cite[Lemma 2.10]{DoRy22x}). This is equivalent to requiring that the map $\sigma^{z}$ is a homomorphism, as we show next in the more general case.
\begin{prop}\label{prop_omo}
    Let $(S, +, \circ)$ be a dual weak brace and $z\in S$. Then, $\sigma^{z}: \left(S,\circ\right)\to \Map(S)$ is a homomorphism if and only if $a \circ z = z+a$, for every $a \in S$. 
    \begin{proof}
If $a \in S$ and $\sigma^{z}$ is a homomorphism, by $\sigma^z_{z^-}\sigma^z_{z^-\circ z}(a)=\sigma^z_{z^- \circ z^-\circ z}(a)$, using \eqref{pro_dual_idemp}, we get
$
z^-\circ(-z+a\circ z) \circ z=
z^- \circ a \circ z$.
Thus,  by the last identity and \eqref{pro_dual_idemp}, we obtain 
\begin{align*}
    z+a &=z+z\circ \left(z^-\circ a \circ z\right) \circ z^-=z+z\circ z^-\circ(-z+a\circ z) \circ z\circ z^-\\
    &=z-z+a \circ z=a \circ z.
\end{align*}
Conversely, if $a,b \in S$, we have
\begin{align*}
    \sigma_a^z\sigma_b^z(c)
    &= - a - z + z + a\circ \left(- b - z + z + b\circ c\right)\\
    &= - a - z +z + a  -a\circ b + a -a + a\circ b\circ c&\mbox{by \eqref{pro_dual_idemp}}\\
    &= - a\circ b\circ z + a\circ b\circ c\circ z\\
   &= \sigma_{a\circ b}^z(c).
\end{align*}
Therefore, the claim follows.
    \end{proof}
\end{prop}

\noindent In the case of a dual weak brace $S$, even if $\sigma^{z}: \left(S,\circ\right)\to \Map(S)$ is a homomorphism, in general, $\sigma_a^z$ does not coincide with $\lambda_a$, since $\sigma_a^z(b)=\lambda_a(b)+z\circ z^-$. In the study of deformed solutions, the following question arises.
\begin{que}\label{que-equiv-sol}
    Let $(S,+, \circ)$ be a dual weak brace. For which parameters $z,w\in S$, are the deformed solutions $r_z$ and $r_w$ 
    equivalent?
\end{que}
We recall that two solutions $r$ and $s$ on two sets $S$ and $T$, respectively, are said to be \emph{equivalent} if there exists a bijective map $\varphi:S \to T$ such that $(\varphi \times \varphi)r = s(\varphi \times \varphi)$ (see \cite{ESS99}).
In the direction of \cref{que-equiv-sol}, in the context of skew braces, in \cite[Example 2.14]{DoRy22x} and \cite[Example 2.15]{DoRy22x} one can find instances of different parameters which give rise to non-equivalent deformed solutions.  Here, the following example shows that in the case of dual weak braces,  even deformed solutions by idempotents are not equivalent in general. 
\begin{ex}
Let $X= \{e, x, y\}$ and $\left(S,\circ\right)$ the
commutative inverse monoid on $X$ with identity $e$ satisfying the relations $x\circ x = y\circ y = x$
and $x\circ y = y$. Note that $a^-=a$, for every $a \in S$.
Consider the trivial weak brace on $S$, namely $a + b = a\circ b$, for all $a,b\in S$. Then, by \cref{th: deform_dualweak}, we have two solutions $r_e=r$ and $r_x$ related to the two idempotents $e$ and $x$, respectively, for which the maps $\sigma^e$ and $\sigma^x$ are explicitly given by
$\sigma^e_a\left(b\right)
    = \lambda_a\left(b\right) + e=\lambda_a(b)
    = a\circ a\circ b$
and $
    \sigma^{x}_a\left(b\right)
    = \lambda_a\left(b\right) + x 
    = -a + a\circ b + x
    = a\circ a\circ b\circ x.
$
If the two solutions $r_{x}$ and $r_{e}$ were equivalent via a bijection $\varphi:S\to S$, then, in particular, we would have that 
$\varphi\left(a\circ a\circ b\circ x\right) = \varphi\left(a\right)\circ \varphi\left(a\right)\circ \varphi\left(b\right)$, for all $a,b\in S$.
Thus, if $a = b = e$ we have that
$\varphi\left(x\right) = \varphi\left(e\right)\circ \varphi\left(e\right)\circ \varphi\left(e\right) = \varphi\left(e\right)$, a contradiction.
\end{ex}

\medskip

Observe that if $z,w\in S$ give rise to two deformed solutions $r_z$ and $r_w$, respectively, and there exists $\varphi\in\Aut\left(S,+,\circ\right)$ such that $\varphi\left(z\right) = w$, then $r_z$ and $r_w$ are trivially equivalent via $\varphi$.  In the special case of a two-sided skew brace, such a map $\varphi$ exists when $z$ and $w$ are in the same conjugacy class, as we show in the next result.

\begin{prop}\label{conjugacy}
    Let $\left(B,+,\circ\right)$ be a two-sided skew brace and $z,w\in B$ belonging to the same conjugacy class in $\left(B, \circ\right)$. Then, the deformed solutions $r_z$ and $r_w$ are equivalent. 
     \begin{proof}
Due to \cref{two-sided}, $r_z$ and $r_w$ are deformed solutions on $B$. By \cite[Proposition 2.3]{Trap22x} and \cite[Lemma 4.1]{Na19}, all the inner automorphisms of $(B, \circ)$ are skew brace automorphisms of $B$. 
By the assumption, there exist $c\in B$ such that $w = c^-\circ z\circ c$, thus $r_z$ and $r_w$ are equivalent via the inner automorphism $\varphi_c$ given by $\varphi_c\left(a\right) = c^-\circ a\circ c$, for any $a\in B$. In particular, $\left(\varphi_c\times \varphi_c\right)r_{z} =  r_{w}\left(\varphi_c \times \varphi_c\right)$.  
     \end{proof}
\end{prop} 

Note that \cref{conjugacy} could is also true in the context of dual weak braces whenever the map $\varphi_c$ is bijective.
\medskip

\begin{rem}
The converse of \cref{conjugacy} is not true. To show this, it is enough to consider the trivial brace on the cyclic group $\mathbb{Z}/2\mathbb{Z}$. Then, the solution $r_0$ coincides with the solution $r_1,$ but $0$ and $1$ trivially belong to different conjugacy classes.
\end{rem}

\bigskip

\section{Structural properties of the distributor}
In this section, we focus on the distributor of any dual weak brace $(S,+, \circ)$ and  highlight some properties for the special case of braces.
\medskip

According to \cref{two-sided}, if $B$ is a skew brace, then $\mathcal{D}_r(B)=B$ if and only if $B$ is a two-sided skew brace. The other limit case is when there exists only the trivial deformation, in other words, $\mathcal{D}_r(B) = \{0\}$. We give some examples below.
\begin{exs}\label{exs:ex:dis}\hspace{1mm}
    \begin{enumerate}
        \item\label{exs:ex:dis:eq1} 
        Let $\left(B, +, \circ\right)$ be the brace on $\left(\mathbb{Z}, +\right)$ with $a\circ b = a + \left(-1\right)^ab$, for all $a,b\in \mathbb{Z}$, cf. \cite[Proposition 6]{Ru07-cyc}. Then, $\mathcal{D}_r\left(B\right) = \{0\}$ (it is enough to choose $a,b\in\mathbb{Z}$ both odd). 
        
    
    \item Let $\left(B, +, \circ\right)$ be the brace on $\left(\mathbb{Z}, \circ\right) = \left< \,g\,\right>$ with $g^k + g^l = g^{k +\left(-1\right)^kl}$, for all $k,l\in \mathbb{Z}$. Then, since $B$ is a two-sided skew brace, $\mathcal{D}_r\left(B\right) = B$.
    \end{enumerate}
\end{exs}

\medskip

 The following are examples of skew braces in which $\mathcal{D}_r(B)$ is not trivial.
 \begin{ex}\label{ex_7}
Let $B$ be a two-sided skew brace and $C$ be a skew brace such that $\dr{C}=\{0\}$. Then, $\dr{B\times C} = B\times \{0\}$.
\end{ex}

\begin{ex}\label{exs:6}
Let $n\in\mathbb{N}_0$ and let us denote by $A_n$ the brace with additive group $(\mathbb{Z}/n\mathbb{Z},+)$ and multiplication given by $a\circ b=a+(-1)^ab,$ for all $a,b\in \mathbb{Z}/n\mathbb{Z}.$ Then, it is a routine computation to check that $z\in \dr{A_n}$ if and only if $4z\equiv0\pmod{n}.$ Thus, if $n=0,$ we get $\dr{A_0}=\{0\}$ (cf. \cref{exs:ex:dis}-1.). If $n\geq 1$, 
    \begin{itemize}
\item[-] if $\mathrm{gcd}(4,n)=1,$ then $\dr{A_n}=\{0\}$,
\item[-] if $\mathrm{gcd}(4,n)=2,$ then $\dr{A_n}=\{0,\frac{n}{2}\}$,
\item[-] if $\mathrm{gcd}(4,n)=4,$ then $\dr{A_n}=\{0,\frac{n}{4},\frac{n}{2}, \frac{3n}{4}\}$.
    \end{itemize}
\end{ex}

\smallskip

\medskip

More generally, one can prove the following results related to the distributor of any dual weak brace. Let us first recall that an inverse subsemigroup $I$ of an inverse semigroup $S$ is \emph{full} if $\E(S) \subseteq I$ (see \cite[p. 19]{Law98}).
\begin{prop} \label{prop_dr_dual}
    Let $(S,+, \circ)$ be a dual weak brace. Then, $ \mathcal{D}_r(S)$ is a full inverse subsemigroup of the Clifford semigroup $(S, \circ)$ containing the center $\zeta(S, \circ)$ of $(S, \circ)$. 
   \begin{proof}
   Initially, it holds $E\left(S\right) \subseteq \mathcal{D}_r(S)$. Moreover, if $a, b \in S$ and $z_1, z_2 \in \mathcal{D}_r(S)$, by \eqref{cond_zmeno}, we get
   \begin{align*}
       (a+b) \circ (z_1\circ z_2)&= \left(a \circ z_1+\left(z_1^-+b\right) \circ z_1\right) \circ z_2\\
       &=a \circ z_1 \circ z_2 +\left(z_2^-+\left(z_1^-+b\right) \circ z_1\right) \circ z_2\\
       &=a \circ z_1 \circ z_2 +\left(z_2^-\circ z_1^-\circ z_1+\left(z_1^-+b\right) \circ z_1\right) \circ z_2 &\mbox{by \eqref{pro_dual_idemp}}\\&=a \circ z_1 \circ z_2 + \left(\left(z_1 \circ z_2\right)^-+b\right) \circ z_1 \circ z_2,
   \end{align*}
   namely $z_1 \circ z_2 \in \mathcal{D}_r(S)$. Besides,  by \eqref{pro_dual_idemp},
   \begin{align*}
       (a+b) \circ z_1^-&=\left(a\circ z_1^-\circ z_1+\left(z_1^- -z_1^-+ b\circ z_1^-\right)\circ z_1\right)\circ z_1^-\\
       &\underset{\eqref{cond_zmeno}}{=}\left(a \circ z_1^--z_1^-+b \circ z_1^-\right) \circ z_1 \circ z_1^-\\&=a \circ z_1^--z_1^-+b \circ z_1^-,
   \end{align*}
   i.e., $z_1^- \in \mathcal{D}_r(S)$. Besides, if $z\in \zeta(S, \circ)$, then
$
(a+b)\circ z=z\circ a-z+z\circ b=a\circ z-z+b\circ z,
$
  for all $a,b\in S$, i.e., $z \in \mathcal{D}_r(S)$. Therefore, the claim follows.
   \end{proof}
\end{prop}

\medskip

Clearly, in the case of a skew brace $B$, $\mathcal{D}_r(B)$ is a subgroup of $(B, \circ)$ containing the center $\zeta(B, \circ)$ of the group $(B, \circ)$.

\medskip

In general, $\mathcal{D}_r(S)$ is not an inverse subsemigroup of the additive semigroup, unless we get into particular cases. 
\begin{prop}
Let $(S,+, \circ)$ be a dual weak brace in which the Clifford semigroup $(S,+)$ is commutative, then $\mathcal{D}_r(S)$ is a two-sided dual weak subbrace of $S$.
\begin{proof}
By \cref{prop_dr_dual}, it is enough to show that $\mathcal{D}_r(S)$ is an inverse subsemigroup of $(S,+)$. Clearly, $\E(S,+) \subseteq \mathcal{D}_r(S)$ because $\E(S, \circ)=\E(S,+)$. Moreover, if $x,y\in \mathcal{D}_r(S)$ and $a,b\in S$, we have
$$
\begin{aligned}
(a+b)\circ (x+y)&=a\circ x-x+b\circ x-b-a+a\circ y-y+b\circ y\\ &=a\circ x-a+a\circ y-(x+y)+b\circ x-b+b\circ y\\ &=a\circ(x+y)-(x+y)+b\circ (x+y),
\end{aligned}
$$
and thus $x+y\in \mathcal{D}_r(S)$. Now, by \cref{lemma_dual_weak}-\ref{lem:5}, 
$$
\begin{aligned}
(a+b)\circ (-x)&=a+b-(a+b)\circ x+a+b=a+b-a\circ x+x-b\circ x+a+b\\ &=a-a\circ x+a +x+b-b\circ x+ b=a\circ (-x)-(-x)+b\circ (-x),
\end{aligned}
$$
and so $-x\in \mathcal{D}_r(S)$. Therefore, the claim follows.
\end{proof}
\end{prop}

\medskip


\begin{rem}\label{cor_brace}
    As a consequence, if $B$ is a brace, its right distributor is a two-sided subbrace of $B$.
\end{rem}

\begin{rem}
    Observe that if $(S,+, \circ)$ is a dual weak brace and $\lambda_z\left(\mathcal{D}_r(S)\right)\subseteq \mathcal{D}_r(S)$, for every $z \in \mathcal{D}_r(S)$,  then $\mathcal{D}_r(S)$ is an inverse subsemigroup of the additive Clifford semigroup $(S,+)$.
    Indeed, it is enough to observe that if $z,w\in \mathcal{D}_r(S)$, then  $z+w = z \circ \lambda_{z^-}\left(w\right)$ and that $-z=\lambda_{z}\left(z^-\right)$.    
\end{rem}
\medskip
As is usual in ring theory, in any brace $\left(B, +, \circ\right)$ we can define the binary operation $a\cdot b := -a + a\circ b - b$, for all $a,b\in B$, cf. \cite{CeSmVe19}. In particular, by \cref{cor_brace}, it follows  that $\left(\mathcal{D}_r(B), +, \cdot\right)$ is a radical ring contained in $B$. Moreover, observe that
 \begin{align}\label{puntino}
     \forall \, a, b \in B \quad z \in \mathcal{D}_r(B) \iff (a+b)\cdot z=a \cdot z+b \cdot z.
 \end{align}

 The following result describes all the parameters giving rise to a deformed solution in a left brace.
 The proof is essentially obtained by extracting the key equalities contained in the proof of \cite[Theorem 1.1]{Lau19}.
\begin{theor}
Let $\left(B,+,\circ\right)$ be a brace. Then, it holds that
\begin{align*}
    \mathcal{D}_r(B) = \{z\in B\ |\ \forall\, a,b\in B\quad (a\cdot b)\cdot z = a\cdot (b\cdot z) \}.
\end{align*}
\begin{proof}
If $z\in \mathcal{D}_r(B)$, the claim is proven by describing $\circ$ in terms of $\cdot$ in the associativity condition of the $\circ$. For the other inclusion, by making explicit the equality $(a\cdot b)\cdot z=a\cdot (b\cdot z)$ and multiplying both sides by $a^-$, we get
$$
(b+a^-\circ(-b))\circ z=b\circ z-z+2\left(a^-\circ z\right)-a^-\circ b\circ z,
$$
for all $a,b \in B$. Now, since by the proof of \cite[Proposition 3.1 (ii)]{Lau19}, $(-x)\circ z =2z-x\circ z,$ for all $z,x\in B$ such that $(x\cdot(-x))\cdot z=x\cdot((-x)\cdot z),$ we get that, for all $a, b \in B$,
$2\left(a^- \circ z\right)-a^-\circ b \circ z= a^- \circ (2z-b \circ z)=a^- \circ (-b) \circ z,$
and so
$$
\left(b+a^-\circ(-b)\right)\circ z=b\circ z-z+\left(a^-\circ(-b)\right)\circ z,
$$
i.e. $z\in \dr B.$
\end{proof}
 \end{theor}   


\medskip

In light of the previous proposition, if $B$ is a brace, we automatically obtain examples of right modules on the ring $\dr B$ having $B$ as underlying set. Besides, when $\dr B$ is not trivial, we also get non-trivial instances of $R$-module braces, since $\lambda_{a}\left(B\right)\subseteq\Aut_{\dr B}\left(B\right)$, for every $a\in B$ (see \cite[Definition 2]{Del23}).

\begin{cor}
Any brace $\left(B,+,\circ\right)$ is a 
$\dr B$-module brace.
\end{cor}
\medskip

Now, it becomes natural to wonder if the distributor is an ideal. We recall that a subset $I$ of a skew brace $B$ is a \emph{left ideal} if it is both a normal subgroup of $(B, +)$ and $\lambda_a(I) \subseteq I$, for every $a \in B$. Moreover, a left ideal $I$ of $B$ is an \emph{ideal} of $B$ if it is a normal subgroup of $(B, \circ)$.  
Equivalently, according to \cite[Lemmas 1.8--1.9]{CeSmVe19}, $I$ is a left ideal if and only if $ B\cdot I \subseteq I$ and it is an ideal of $B$ if and only if also
$I\cdot B \subseteq I.$ 

Recalling that a non-trivial brace is right nilpotent  of index $3$ if $B^{2}\cdot B = \{0\}$ and left nilpotent of index $3$ if $B\cdot B^{2} = \{0\}$,
by \eqref{puntino} we have the following result.
\begin{prop}
    Let $B$ be a brace. If $B$ is right nilpotent of index $3$, then $\dr{B}$ is a left ideal of $B$. If $B$ is left nilpotent of index $3$, then $z\cdot b\in\dr{B}$, for all $b\in B$ and $z\in \dr{B}$.
\end{prop}
\begin{ex}
    If we consider the brace $A_6$ as in \cref{exs:6}, it is easy to check that $A_6$ is right nilpotent of index $3$, hence $\dr{A_6}$ is a left ideal of $A_6$.
\end{ex}

In the following, we characterize, in general, when the distributor is an ideal for the braces $A_n$.
\begin{prop}\label{prop:class}
Let $n \in \mathbb{N}_0$ and $A_n$ be the brace defined in \cref{exs:6}. Then,  $\dr {A_n}$ is a left ideal. Moreover, $\dr {A_n}$ is an ideal of $A_n$ if and only if $n\in\{2,4\}$, $\mathrm{gcd}(4,n)=1$, or $8\mid n.$
\begin{proof}
The first statement holds by observing that $k\cdot d = \left(\left(-1\right)^k-1\right)d\in \dr{A_n}$, for all $k\in A_n$ and $d\in \dr{A_n}$.\\ 
For all $k\in A_n$ and $d\in \dr {A_n}$, we have:
\begin{equation}\label{cases:class}
d\cdot k=\left(\left(-1\right)^d-1\right)k=\begin{cases}0\  &\text{if}\  2\mid d\\ -2k\  &\text{if}\  2\nmid d\end{cases}   \,. 
\end{equation}
Assume that $\dr {A_n}$ is an ideal of $A_n$ and let us break down our consideration to the following cases.
\begin{itemize}
    \item[-] If $\mathrm{gcd}(n,4)=1,$ then $\dr{A_n}=\{0\}$ and $n$ can be any number coprime with $4.$ 
    \item[-] If $\mathrm{gcd}(n,4)=2,$ then $\dr{A_n}=\{0,\frac{n}{2}\}$ and $\frac{n}{2}\cdot k =0 \pmod n$ or $\frac{n}{2}\cdot k =\frac{n}{2}\pmod n.$ Since $\mathrm{gcd}(n,4)=2,$ then $2\nmid \frac{n}{2}$ and so, by \eqref{cases:class}, since $\dr {A_n}$ is an ideal, we have that $-2k\equiv\frac{n}{2}\pmod n$ or $-2k\equiv0\pmod n.$ The first congruence leads to a contradiction with $\mathrm{gcd}(n,4)=2,$ while the second gives $n=2.$
    \item[-] If $\mathrm{gcd}(n,4)=4,$ then $\dr{A_n}=\{0, \frac{n}{4}, \frac{n}{2}, \frac{3n}{4}\}$ and we can consider two cases, i.e., when $2\mid d$ for all $d\in \dr{A_n}$ or when there exists $d\in \dr{A_n}$ such that $2\nmid d.$ In the first case, we obtain that $8\mid n,$ since $2\mid \frac{n}{4}$. In the second one, we get $d \cdot k=-2k \equiv \frac{ni}{4}$, with $i \in \{0,1,2,3\}$
    and for all $k\in A_n$. If $i=1,3$ we get a contradiction with $2\nmid d$.  Besides, the congruence with $i=0$ is satisfied if and only if $n=2,$ but $\mathrm{gcd}(n,4)=4.$ The third congruence for $i=2$ and $k=1$ implies $n=4.$
\end{itemize}
In the opposite direction, in the cases $n=2,4$ or $\mathrm{gcd}(4,n)=1,$ the distributor is a trivial ideal. The claim in the case when $8\mid n$ follows directly from \eqref{cases:class}.
\end{proof}
\end{prop}
\medskip

If $(B, +, \circ)$ is a skew brace, the  group $\mathcal{D}_r(B)$ can be related to $\Fix(B)$ and to its annihilator  $\Ann\left(B\right)$. 
According to \cite{CeSmVe19}, 
$\Fix\left(B\right) = \{a\in B \ | \ \forall \ x\in B\quad \lambda_x\left(a\right) =  a \}$
and it is a left ideal of $B$.
Besides, the \emph{annihilator} of $B$ is an ideal of $B$ defined by 
$\Ann\left(B\right)= \Soc\left(B\right)\cap \zeta\left(B, \circ \right)$, where $\Soc\left(B\right) =\{a \in B \, \mid \, \forall \ b\in B\quad a + b = a\circ b \} \cap \zeta(S, +)$, see  \cite{CCoSt19}. 
It is a routine computation to check the following inclusion.

\begin{prop}
    Let $(B,+, \circ)$ be a skew brace. Then, $\Ann(B) \subseteq \Fix(B) \subseteq \mathcal{D}_r(B)$.
\end{prop}

\bigskip

\section{Bites of parameters}
In this section, regarding a dual weak brace $S$ as a strong semilattice of skew braces $B_\alpha$, we analyze how the entire distributor of $S$ interacts with the distributor of each $B_\alpha$. In addition, we show when a deformed solution on $S$ is the strong semilattice of deformed solutions on $B_\alpha$.

\medskip

Hereinafter, through the section, $S$ will be seen as a strong semilattice $[Y, B_\alpha, \phi_{\alpha, \beta}]$. First, in the following, we show when  $\mathcal{D}_r(S)$ is the disjoint union of each $\mathcal{D}_r(B_\alpha)$.

\begin{theor}\label{prop_dr(S)}
Let $S$ be a dual weak brace, then $\mathcal{D}_r(S)\subseteq\mathop{\dot{\bigcup}}\limits_{\alpha\in Y} \mathcal{D}_r(B_\alpha)$.\\ 
Moreover, $\mathcal{D}_r(S)=\mathop{\dot{\bigcup}}\limits_{\alpha\in Y} \mathcal{D}_r(B_\alpha)$ if and only if $\phi_{\alpha,\beta}\left(\mathcal{D}_r(B_\alpha)\right)\subseteq \mathcal{D}_r(B_\beta)$, for all $\alpha,\beta\in Y$ such that $\beta\leq\alpha$.
\end{theor}
\begin{proof}
If $z\in \mathcal{D}_r(S),$ then there exists $\alpha\in Y$ such that $z\in \mathcal{D}_r(B_\alpha)$ and thus $\mathcal{D}_r(S)\subseteq\mathop{\dot{\bigcup}}\limits_{\alpha\in Y} \mathcal{D}_r(B_\alpha).$\\
Let us assume that $z\in\mathop{\dot{\bigcup}}\limits_{\alpha\in Y}\mathcal{D}_r(B_\alpha)$ and that $\phi_{\alpha,\beta}(\mathcal{D}_r(B_\alpha))\subseteq \mathcal{D}_r(B_\beta),$ for all $\alpha,\beta\in Y$ such that $\beta \leq \alpha$. Then, there exists $\gamma\in Y$ such that $z\in \mathcal{D}_r(B_\gamma)$, and, for all $a\in B_\alpha$ and $b\in B_\beta$, we have
$$
\begin{aligned}
(a+b)\circ z&=(\phi_{\alpha,\alpha\beta\gamma}(a)+\phi_{\beta,\alpha\beta\gamma}(b))\circ \phi_{\gamma,\alpha\beta\gamma}(z)\\ &=\phi_{\alpha,\alpha\beta\gamma}(a)\circ \phi_{\gamma,\alpha\beta\gamma}(z)- \phi_{\gamma,\alpha\beta\gamma}(z)+\phi_{\beta,\alpha\beta\gamma}(b)\circ \phi_{\gamma,\alpha\beta\gamma}(z)\\
&=a\circ z-z+b\circ z,
\end{aligned}
$$
since $\phi_{\gamma,\alpha\beta\gamma}\left(\mathcal{D}_r(B_\gamma))\subseteq \mathcal{D}_r(B_{\alpha\beta\gamma}\right).$ Hence, 
$z\in \mathcal{D}_r(S).$\\
Conversely, if $\mathcal{D}_r(S) = \mathop{\dot{\bigcup}}\limits_{\alpha\in Y} \mathcal{D}_r(B_\alpha),$ then for all $\alpha,\beta\in Y,$ $a,b\in B_\beta$ and $z\in \mathcal{D}_r(B_\alpha)$ such that $\beta\leq\alpha,$ we get that
$$(a+b)\circ \phi_{\alpha,\beta}(z)=(a+b)\circ z=a\circ z-z+b\circ z=a\circ \phi_{\alpha,\beta}(z)-\phi_{\alpha,\beta}(z)+b\circ \phi_{\alpha,\beta}(z),
$$
since $z\in \mathcal{D}_r(S).$ Thus, 
$\phi_{\alpha,\beta}(\mathcal{D}_r(B_\alpha))\subseteq \mathcal{D}_r(B_\beta).$
\end{proof}

\medskip

In the following example, $\mathcal{D}_r(S)$ is not the union of distributors.
  \begin{ex}
        Let $Y=\{\alpha, \beta\}$, with $\beta < \alpha$. Considering the cyclic group $C_6:=\left(\mathbb{Z}/6\mathbb{Z},+\right)$, let $B_\alpha$ be the trivial brace on $C_6$, $B_\beta$ the brace $A_6$ given in \cref{exs:6}, and $\varphi:B_\alpha\to B_\beta$ the brace homomorphism given by $\varphi\left(a\right) = 2a$, for all $a\in B_\alpha$. Then, $S=[Y, B_\gamma, \varphi]$ is a dual weak brace. Moreover, $\mathcal{D}_r\left(B_\alpha\right)=B_\alpha$,  $\mathcal{D}_r\left(B_\beta\right)=\{0_\beta, 3_\beta\}$, and $\varphi\left(1_\alpha\right) = 2_\beta\notin \mathcal{D}_r\left(B_\beta\right)$, hence trivially $\varphi\left(\mathcal{D}_r\left(B_\alpha\right)\right) \nsubseteq\mathcal{D}_r\left(B_{\beta}\right)$. Indeed, in this case, $\mathcal{D}_r\left(S\right)=\{\,0_\alpha, 3_\alpha, 0_\beta, 3_\beta\,\}$.
    \end{ex}

\medskip

Note that in the example above $\mathcal{D}_r\left(S\right)$ is not an ideal of $S$ although $\mathcal{D}_r\left(B_\gamma\right)$ is an ideal of each skew brace $B_\gamma$.  We highlight that the notion of ideal has also been given for dual weak braces in \cite{CaMaSt24} and makes use of the definitions of normal subsemigroups of Clifford semigroups. Moreover,  \cite[Theorem 3]{CaMaSt24} is a structure theorem for ideals of a dual weak brace $S=[Y, B_\alpha, \phi_{\alpha, \beta}]$ in terms of the ideals of the skew braces $B_\alpha$. Indeed, as a direct consequence of \cref{prop_dr(S)} and \cite[Theorem 3.2]{CaMaSt24}, we obtain the following result. 
\begin{cor}
    Let $S$ be a dual weak brace. If $\mathcal{D}_r(B_\alpha)$ is an ideal of each skew brace $B_\alpha$, for every $\alpha \in Y$, then $\mathcal{D}_r(S)$ is an ideal of $S$ if and only if $\mathcal{D}_r(S)=\mathop{\dot{\bigcup}}\limits_{\alpha\in Y} \mathcal{D}_r(B_\alpha)$. 
\end{cor}

\medskip


In this part, we compare deformed solutions on a dual weak brace 
$S$ acquired in \cref{th: deform_dualweak} with solutions constructed as a strong semilattice of deformed solutions on skew braces $B_\alpha$. Although in \cref{prop_dr(S)} we characterize when $\mathcal{D}_r(S)=\mathop{\dot{\bigcup}}\limits_{\alpha\in Y} \mathcal{D}_r(B_\alpha)$, it is not guaranteed that, in this case, $r_z$ is a strong semilattice of some deformed solutions on $B_\alpha$, for some $z \in \dr{S}$, since $3.$ of \cref{thm:4.1} is not satisfied, in general. We will show that it is true only for some parameters and if the semilattice is bounded.

\medskip

\begin{defin}\label{def:ssp}
Let $S$ be a dual weak brace. A subset $P\subseteq S$ of $S$ is said to be a  \emph{bite of parameters} if the following hold:
\begin{enumerate}
\item $P\cap B_\alpha=\{p_{\alpha}\}\subseteq \mathcal{D}_r(B_\alpha)$, for every $\alpha\in Y$,
\item 
$a\circ p_{\beta} - a\circ p_\alpha = p_{\beta} - p_\alpha$,
for all $\alpha,\beta\in Y$ such that $\beta \leq \alpha$ and $a \in B_\alpha$.
\end{enumerate}
We will denote the family of all bites of parameters of $S$ by $\mathcal{B}(S)$.
\end{defin}

\begin{lemma}\label{lemma_param}
Let $S$ be a dual weak brace and $P\subseteq{S}$ such that $P\cap B_\alpha=\{p_{\alpha}\}\subseteq \mathcal{D}_r(B_\alpha)$, for every $\alpha\in Y$. If $\phi_{\alpha,\beta}\left(p_\alpha\right)\in P,$ for all $p_\alpha\in P$, $\alpha,\beta\in Y$ such that $\beta \leq \alpha$, then $P\in\mathcal{B}(S)$.
\end{lemma}

\noindent The converse of \cref{lemma_param} is not true, as we show in the following example.
\begin{ex}
    Let $S$ be the dual weak braces given by $Y=\{\alpha,\beta\}$ with $\beta < \alpha$, $B_\alpha=\mathbb{Z}$, $B_\beta=\mathbb{Z}/n\mathbb{Z}$ trivial braces, and $\phi_{\alpha, \beta}:\mathbb{Z}\to \mathbb{Z}/n\mathbb{Z}$  the canonical epimorphism. Thus, $P=\{0_\alpha,2_\beta\}\in \mathcal{B}(S),$ but $\phi_{\alpha, \beta}\left(0_\alpha\right)=0_\beta\not\in P.$
\end{ex}

\begin{rem}
Observe that for any dual weak brace $S$, we have that $\mathcal{B}(S)$ is non-empty. Indeed, the set $\E(S)=\{0_\alpha\ |\ \alpha\in Y\}$ is a bite of parameters, as every homomorphism $\phi_{\alpha, \alpha\beta}$ preserves the identity of each skew brace.
\end{rem}

\begin{ex}
Let $S=[Y, U_n, \phi_{n, m}]$ be the dual weak brace in \cref{ex_u}, $M$ the maximum in the set $Y$, and $a\in U_M$ a fixed element. Then, the set $\{\phi_{M,m}(a)\ |\ m\in Y\}\in \mathcal{B}(S)$.
\end{ex}

\smallskip

In the following, to avoid the overloading notation, by $r_{p_\alpha}$ we mean the deformed solution on the skew brace $B_\alpha$ by a parameter $p_\alpha \in \mathcal{D}_r\left(B_\alpha\right)$.
\begin{lemma}\label{lemma_solu}
Let $S$ be a dual weak brace and $r_{p_\alpha}$ a deformed solution on $B_\alpha$, for every $\alpha \in Y$. Then, the map $t:S\times S\to S\times S$ defined by
$$
t(x,y):=r_{p_{\alpha\beta}}(\phi_{\alpha,\alpha\beta}(x),\phi_{\beta,\alpha\beta}(y)), 
$$
for all $x\in B_\alpha$, $y\in B_\beta$, is a strong semilattice of the solutions $r_{p_\alpha}$ if and only if $\{p_\alpha\ |\ \alpha\in Y\}\in \mathcal{B}(S).$ 
\end{lemma}
\begin{proof}
Let us assume that $\{p_\alpha\ |\ \alpha\in Y\}\in \mathcal{B}(S).$ Then, to get the claim, we have to check $3.$ of \cref{thm:4.1}, i.e., 
\begin{equation}\label{eq1Lem50}
\phi_{\alpha,\beta}(a)\circ p_\beta-\phi_{\alpha,\beta}(a\circ p_\alpha)=\phi_{\alpha,\beta}(a\circ b)\circ p_\beta-\phi_{\alpha,\beta}(a\circ b\circ p_\alpha),
\end{equation}
for all $\alpha, \beta \in Y$ such that $\beta \leq \alpha$ and for all $a,b\in B_\alpha$.
Using $2.$ of \cref{def:ssp}, we get that
$$
\begin{aligned}
\phi_{\alpha,\beta}(a)\circ p_\beta-\phi_{\alpha,\beta}(a\circ p_\alpha)&= p_\beta- \phi_{\alpha,\beta}(p_\alpha)= \phi_{\alpha,\beta}(a \circ b)\circ p_\beta-\phi_{\alpha,\beta}(a\circ b)\circ \phi_{\alpha,\beta}(p_\alpha)\\ &=\phi_{\alpha,\beta}(a\circ b)\circ p_\beta-\phi_{\alpha,\beta}(a\circ b\circ p_\alpha).
\end{aligned}
$$
Other way, let us assume that $r_{p_{\alpha\beta}}$ is the strong semilattice of the solutions $r_{p_\alpha}$ and consider $P=\{p_\alpha\ |\ \alpha\in Y\}.$ Clearly, $P\cap B_\alpha=\{p_{\alpha}\}\subseteq \mathcal{D}_r(B_\alpha)$, for every $\alpha\in Y.$ By the previous part, if $a \in B_\alpha$,
by taking $b=a^-$ in \eqref{eq1Lem50}, we obtain property 2. of \cref{def:ssp}.
\end{proof}

\medskip

In the following result, we will denote by $1$ the join of the semilattice $Y$, whenever it is bounded.

\begin{theor}
Let $S$ be a dual weak brace and $z\in \mathcal{D}_r(S).$ 
\begin{enumerate}
\item If $Y$ is not bounded or $z\not\in B_1$, then the deformed solution $r_z$ is not a strong semilattice of  the solutions defined on $B_{\alpha}$, for every $\alpha \in Y$.
\item  If $Y$ is bounded and $r_z$ is a strong semilattice of solutions, then $z\in B_1$ and there exists $P=\{\phi_{1, \alpha}(z) \, \mid \, \alpha \in Y\}\in \mathcal{B}(S)$ such that $r_z$ is the strong semilattice of solutions $r_{\phi_{1, \alpha}(z)}$ on each skew brace $B_\alpha$, for every $\alpha \in Y$.
\end{enumerate}
\begin{proof} \hspace{1mm}
\begin{enumerate}
\item 
Let $z\in B_{\alpha}$ for some $\alpha\in Y.$ Observe that if $Y$ is not bounded or $\alpha\not=1,$ then there exists $\beta\in Y$ such that $\beta\alpha\not=\beta.$ In that case, for all $a,b\in B_\beta$,
$
r_z(a,b)\subseteq B_{\beta\alpha}\times B_{\beta\alpha}\not= B_\beta\times B_\beta,$ and thus $r_{z\mid B_\beta \times B_\beta}$ is not a well-defined solution on $B_\beta.$ Consequently, $r_z$ is not a strong semilattice of solutions on $B_\alpha$, for every $\alpha \in Y$.

\item Clearly, $z\in B_1$ by the previous point. Let us consider $P=\{\phi_{1,\alpha}(z)\ |\ \alpha\in Y\}.$ Then, one can easily check that $P\cap B_\alpha=\{\phi_{1,\alpha}(z)\} \subseteq \mathcal{D}_r\left(B_\alpha\right)$, for every $\alpha \in Y$. Moreover, if $\alpha,\beta\in Y$ are such that $\beta\leq \alpha,$ then $\phi_{\alpha,\beta}\phi_{1,\alpha}(z)=\phi_{1,\beta}(z)\in P.$ Thus, by \cref{lemma_param}, $P\in\mathcal{B}(S).$ Furthermore, if $\alpha,\beta\in Y$ and $a\in B_\alpha$, $b\in B_\beta$, 
\begin{align*}
    r_z(a,b)=r_{\phi_{1, \alpha \beta}(z)}\left(\phi_{\alpha, \alpha\beta}(a), \phi_{\alpha, \alpha \beta}(b) \right).
\end{align*}
Finally, by \cref{lemma_solu}, we get the claim.
\end{enumerate}
\end{proof}
\end{theor}

\begin{cor}
Let $S$ be a dual weak brace, $P\in \mathcal{B}(S)$, and $z\in P$. Then, $r_z$ is the strong semilattice of deformed solutions on $B_\alpha$ through parameters in $P$ if and only if $Y$ is bounded and $z\in B_1.$ 
\end{cor}


\bigskip

\section*{Acknowledgements}
\noindent The authors thank an anonymous referee for addressing their attention to papers \cite{BeDr82} and \cite{Ku74}.
They also thank Ivan Kaygorodov for his kindness in guiding them to find some literature.

\noindent This work was partially supported by the Dipartimento di Matematica e Fisica ``Ennio De Giorgi'' - Università del Salento.  The first author was supported by the “HaMMon (Hazard Mapping and vulnerability
Monitoring) - Spoke 2” project. The second author was supported by  EPSRC grant EP/V008129/1. The first and the third authors are members of GNSAGA (INdAM) and of the non-profit association ADV-AGTA.

\bigskip 
\bibliography{bibliography}

\def\cprime{$'$}
\begin{thebibliography}{10}
\expandafter\ifx\csname url\endcsname\relax
  \def\url#1{\texttt{#1}}\fi
\expandafter\ifx\csname urlprefix\endcsname\relax\def\urlprefix{URL }\fi

\bibitem{And48}
V.~A. Andrunakievi\v{c}, Semiradical rings, Izv. Akad. Nauk SSSR Ser. Mat. 12
  (1948) 129--178.

\bibitem{BaGu22}
V.~G. Bardakov, V.~Gubarev, Rota-{B}axter groups, skew left braces, and the
  {Y}ang-{B}axter equation, J. Algebra 596 (2022) 328--351.
\newline\urlprefix\url{https://doi.org/10.1016/j.jalgebra.2021.12.036}

\bibitem{Baxter}
R.~J. Baxter, Partition function of the eight-vertex lattice model, Annals of
  Physics 70~(1) (1972) 193--228.
\newline\urlprefix\url{https://www.sciencedirect.com/science/article/pii/0003491672903351}

\bibitem{BeDr82}
A.~A. Belavin, V.~G. Drinfel\cprime~d, Solutions of the classical
  {Y}ang-{B}axter equation for simple {L}ie algebras, Funct Anal Its Appl 16
  (1982) 159--180.
\newline\urlprefix\url{https://doi.org/10.1007/BF01081585}

\bibitem{CCoSt19}
F.~Catino, I.~Colazzo, P.~Stefanelli, Skew left braces with non-trivial
  annihilator, J. Algebra Appl. 18~(2) (2019) 1950033, 23.
\newline\urlprefix\url{https://doi.org/10.1142/S0219498819500336}

\bibitem{CCoSt21}
F.~Catino, I.~Colazzo, P.~Stefanelli, Set-theoretic solutions to the
  {Y}ang-{B}axter equation and generalized semi-braces, Forum Math. 33~(3)
  (2021) 757--772.
\newline\urlprefix\url{https://doi.org/10.1515/forum-2020-0082}

\bibitem{CaMaMiSt22}
F.~Catino, M.~Mazzotta, M.~M. Miccoli, P.~Stefanelli, Set-theoretic solutions
  of the {Y}ang-{B}axter equation associated to weak braces, Semigroup Forum
  104 (2022) 228--255.
\newline\urlprefix\url{https://doi.org/10.1007/s00233-022-10264-8}

\bibitem{CaMaSt24}
F.~Catino, M.~Mazzotta, P.~Stefanelli, Solutions of the {Y}ang-{B}axter
  equation and strong semilattices of skew braces, Mediterr. J. Math. 21~(2)
  (2024) Paper No. 67, 22.
\newline\urlprefix\url{https://doi.org/10.1007/s00009-024-02611-6}

\bibitem{CeSmVe19}
F.~Ced\'{o}, A.~Smoktunowicz, L.~Vendramin, Skew left braces of nilpotent type,
  Proc. Lond. Math. Soc. (3) 118~(6) (2019) 1367--1392.
\newline\urlprefix\url{https://doi.org/10.1112/plms.12209}

\bibitem{ClPr61}
A.~H. Clifford, G.~B. Preston, The algebraic theory of semigroups. {V}ol. {I},
  Mathematical Surveys, No. 7, American Mathematical Society, Providence, R.I.,
  1961.

\bibitem{Del23}
I.~Del~Corso, Module braces: relations between the additive and the
  multiplicative groups, Ann. Mat. Pura Appl. (4) 202~(6) (2023) 3005--3025.
\newline\urlprefix\url{https://doi.org/10.1007/s10231-023-01349-4}

\bibitem{DoRy22x}
A.~Doikou, B.~Rybo{\l}owicz, Novel non-involutive solutions of the
  {Y}ang-{B}axter equation from (skew) braces, {J}. {L}ondon {M}ath. {S}oc.,
  {I}n press.

\bibitem{Dr92}
V.~G. Drinfel\cprime~d, On some unsolved problems in quantum group theory, in:
  Quantum groups ({L}eningrad, 1990), vol. 1510 of Lecture Notes in Math.,
  Springer, Berlin, 1992, pp. 1--8.
\newline\urlprefix\url{https://doi.org/10.1007/BFb0101175}

\bibitem{Et03}
P.~Etingof, Geometric crystals and set-theoretical solutions to the quantum
  {Y}ang-{B}axter equation, Comm. Algebra 31~(4) (2003) 1961--1973.
\newline\urlprefix\url{https://doi.org/10.1081/AGB-120018516}

\bibitem{ESS99}
P.~Etingof, T.~Schedler, A.~Soloviev, Set-theoretical solutions to the quantum
  {Y}ang-{B}axter equation, Duke Math. J. 100~(2) (1999) 169--209.
\newline\urlprefix\url{http://dx.doi.org/10.1215/S0012-7094-99-10007-X}

\bibitem{EtSoGu01}
P.~Etingof, A.~Soloviev, R.~Guralnick, Indecomposable set-theoretical solutions
  to the quantum {Y}ang-{B}axter equation on a set with a prime number of
  elements, J. Algebra 242~(2) (2001) 709--719.
\newline\urlprefix\url{https://doi.org/10.1006/jabr.2001.8842}

\bibitem{GuVe17}
L.~Guarnieri, L.~Vendramin, Skew braces and the {Y}ang-{B}axter equation, Math.
  Comp. 86~(307) (2017) 2519--2534.
\newline\urlprefix\url{https://doi.org/10.1090/mcom/3161}

\bibitem{Ho95}
J.~M. Howie, Fundamentals of semigroup theory, vol.~12 of London Mathematical
  Society Monographs. New Series, The Clarendon Press, Oxford University Press,
  New York, 1995, oxford Science Publications.

\bibitem{KoTr20}
A.~Koch, P.~J. Truman, Opposite skew left braces and applications, J. Algebra
  546 (2020) 218--235.
\newline\urlprefix\url{https://doi.org/10.1016/j.jalgebra.2019.10.033}

\bibitem{Ku74}
A.~G. Kurosh, Lectures of the 1969--1970 Academic Year, Nauka, Moscow, 1974 (in
  Russian).

\bibitem{Lau19}
I.~Lau, An associative left brace is a ring, J. Algebra Appl. 19~(9) (2020)
  2050179, 6.
\newline\urlprefix\url{https://doi.org/10.1142/S0219498820501790}

\bibitem{Law98}
M.~V. Lawson, Inverse semigroups, World Scientific Publishing Co., Inc., River
  Edge, NJ, 1998, the theory of partial symmetries.
\newline\urlprefix\url{https://doi.org/10.1142/9789812816689}

\bibitem{Li98}
F.~Li, {W}eak {H}opf {A}lgebras and {S}ome {N}ew {S}olutions of the {Q}uantum
  {Y}ang--{B}axter {E}quation, J. Algebra 208~(1) (1998) 72--100.
\newline\urlprefix\url{https://doi.org/10.1006/jabr.1998.7491}

\bibitem{LuYZ00}
J.-H. Lu, M.~Yan, Y.-C. Zhu, On the set-theoretical {Y}ang-{B}axter equation,
  Duke Math. J. 104~(1) (2000) 1--18.
\newline\urlprefix\url{http://dx.doi.org/10.1215/S0012-7094-00-10411-5}

\bibitem{Na19}
T.~Nasybullov, Connections between properties of the additive and the
  multiplicative groups of a two-sided skew brace, J. Algebra 540 (2019)
  156--167.
\newline\urlprefix\url{https://doi.org/10.1016/j.jalgebra.2019.05.005}

\bibitem{Pe84}
M.~Petrich, Inverse semigroups, Pure and Applied Mathematics (New York), John
  Wiley \& Sons, Inc., New York, 1984, a Wiley-Interscience Publication.

\bibitem{Ru07}
W.~Rump, Braces, radical rings, and the quantum {Y}ang-{B}axter equation, J.
  Algebra 307~(1) (2007) 153--170.
\newline\urlprefix\url{https://doi.org/10.1016/j.jalgebra.2006.03.040}

\bibitem{Ru07-cyc}
W.~Rump, Classification of cyclic braces, J. Pure Appl. Algebra 209~(3) (2007)
  671--685.
\newline\urlprefix\url{https://doi.org/10.1016/j.jpaa.2006.07.001}

\bibitem{Trap22x}
S.~Trappeniers, On two-sided skew braces, J. Algebra 631 (2023) 267--286.
\newline\urlprefix\url{https://doi.org/10.1016/j.jalgebra.2023.05.003}

\bibitem{Yang}
C.~N. Yang, Some exact results for the many-body problem in one dimension with
  repulsive delta-function interaction, Phys. Rev. Lett. 19 (1967) 1312--1315.
\newline\urlprefix\url{https://link.aps.org/doi/10.1103/PhysRevLett.19.1312}

\end{thebibliography}

\bigskip

\section*{Email addresses}
\begin{itemize}
    \item[]  Marzia Mazzotta\\
    Dipartimento di Matematica e Fisica ``Ennio De Giorgi''
 		\\
		Universit\`{a} del Salento, Via Provinciale Lecce-Arnesano, 
 		73100 Lecce (Italy)\\
   \texttt{marzia.mazzotta@unisalento.it}\\
   ORCID: 0000-0001-6179-9862
\item[] Bernard Rybo\l owicz\\
Department of Mathematics, Heriot-Watt University, 
 		Edinburgh EH14 4AS, and Maxwell Institute for Mathematical Sciences, Edinburgh (UK)\\
    \texttt{B.Rybolowicz@hw.ac.uk} \\
    ORCID: 0000-0002-2894-8288
       \item[] Paola Stefanelli\\
  Dipartimento di Matematica e Fisica ``Ennio De Giorgi''\\
		Universit\`{a} del Salento, Via Provinciale Lecce-Arnesano, 
 		73100 Lecce (Italy) \\
   \texttt{paola.stefanelli@unisalento.it}
   \\
   ORCID: 0000-0003-3899-3151
\end{itemize}

\end{document}